\documentclass[journal,twoside,web]{ieeecolor}
\usepackage{generic}
\usepackage{cite}
\usepackage{amsmath,amssymb,amsfonts}
\usepackage{algorithmic}
\usepackage{graphicx}
\usepackage{textcomp}

\def\qed{\hspace*{\fill}~\IEEEQED\par\endtrivlist\unskip}

\def\Re{\mathbb{R}}

\def\Lemma#1{Lemma~\ref{#1}}

\def\Theorem#1{Thm.~\ref{#1}}

\def\Sec#1{Sec.~\ref{#1}}

\def\notes#1{\marginpar{\tiny #1}\typeout{Notes!
Notes!
Notes!
}}
\renewcommand{\notes}[1]{\typeout{notes!}}
\def\FRAC#1#2#3{\genfrac{}{}{}{#1}{#2}{#3}}
\def\half{{\mathchoice{\FRAC{1}{1}{2}}%
{\FRAC{2}{1}{2}}%
{\FRAC{3}{1}{2}}%
{\FRAC{4}{1}{2}}}}

\newcommand{\tr}{\mbox{tr}}

\def\Re{\field{R}}

\def\Sec#1{Sec.~\ref{#1}}

\def\transpose{{\hbox{\rm\tiny T}}}
\def\clB{{\cal B}}

\def\clL{{\cal L}}
\def\clP{{\cal P}}
\def\clZ{{\cal Z}}

\def\Sec#1{Sec~\ref{#1}}

\def\E{{\sf E}}

\def\Sec#1{Sec.~\ref{#1}}

\def\IEEEQEDclosed{\mbox{\rule[0pt]{1.3ex}{1.3ex}}}
\def\qed{\nobreak\hfill\IEEEQEDclosed}

\def\clZ{{\cal Z}}

\newtheorem{theorem}{Theorem}
\newtheorem{example}{Example}
\newtheorem{definition}{Definition}
\newtheorem{lemma}{Lemma}
\newtheorem{remark}{Remark}
\newtheorem{proposition}{Proposition}

\def\beq{\begin{eqnarray}} 
\def\bc{\begin{center}} 
\def\be{\begin{enumerate}}
\def\bi{\begin{itemize}} 
\def\bs{\begin{small}}
\def\bS{\begin{slide}}
\def\ec{\end{center}} 
\def\ee{\end{enumerate}}
\def\ei{\end{itemize}}
\def\es{\end{small}}
\def\eS{\end{slide}}
\def\eeq{\end{eqnarray}}

\newcommand{\newP}[1]{\medskip\noindent{\bf #1:}}

\newcommand{\ud}{\,\mathrm{d}}

\def\Re{\mathbb{R}}
\def\E{{\sf E}}

\def\clY{{\cal Y}}

\def\Sec#1{Sec.~\ref{#1}}
\def\Thm#1{Thm.~\ref{#1}}
\def\Prop#1{Prop.~\ref{#1}}

\def\clB{{\cal B}}
\def\clD{{\cal D}}

\def\clL{{\cal L}}
\def\clP{{\cal P}}
\def\clZ{{\cal Z}}

\renewcommand{\Re}{\mathbb{R}}

\def\FRAC#1#2#3{\genfrac{}{}{}{#1}{#2}{#3}}

\newcommand{\var}{\text{var}}

\def\clA{{\cal A}}
\def\clB{{\cal B}}

\def\clD{{\cal D}}

\def\clF{{\cal F}}

\def\clL{{\cal L}}

\def\clN{{\cal N}}
\def\clO{{\cal O}}
\def\clP{{\cal P}}

\def\clS{{\cal S}}

\def\clU{{\cal U}}
\def\clV{{\cal V}}

\def\clY{{\cal Y}}
\def\clZ{{\cal Z}}
\def\E{{\sf E}}
\def\bS{\mathbb{S}}
\def\sJ{{\sf J}}

\def\ones{{\sf 1}}
\def\sP{{\sf P}}
\def\tsP{{\tilde{\sf P}}}
\def\tE{{\tilde{\sf E}}}

\def\tp{{\hbox{\rm\tiny T}}}

\def\bmu{{\bar\mu}}

\def\dvar{\operatorname{var}}

\def\chisq{{\chi^2}}
\def\kl{{\sf D}}
\def\tv{{\text{\tiny TV}}}

\def\opt{{\text{\rm (opt)}}}
\renewcommand{\limsup}{\mathop{\operatorname{limsup}}}
\renewcommand{\liminf}{\mathop{\operatorname{liminf}}}

\def\LL{{\sf L}}
\def\II{{\cal E}}
\def\VV{\var^\rho(Y_0(X_0))}

\begin{document}
\title{Variance Decay Property for Filter Stability}
\author{Jin W. Kim, \IEEEmembership{Member, IEEE}, and
	Prashant G. Mehta, \IEEEmembership{Fellow, IEEE}
	\thanks{This work is supported in part by the AFOSR award
          FA9550-23-1-0060 (Mehta) and the DFG grant 318763901/SFB1294 (Kim).
   A portion of this work was conducted during authors' visit to the Issac Newton Institute for Mathematical Sciences, Cambridge.
  }
	\thanks{J. W. Kim is with the Institute of Mathematics at the University of Potsdam
		(e-mail: jin.won.kim@uni-potsdam.de).}
	\thanks{P. G. Mehta is with 
		the Coordinated Science Laboratory and the Department of Mechanical Science
		and Engineering at the University of Illinois at Urbana-Champaign
		(e-mail: mehtapg@illinois.edu).}}
\maketitle

\begin{abstract}
This paper is concerned with the problem of nonlinear (stochastic)
filter stability of a hidden Markov model (HMM) with white noise
observations.  A contribution is the variance decay property
which is used to conclude filter stability.  For this purpose, a new
notion of the Poincar\'e inequality (PI) is introduced for the
nonlinear filter.  PI is related
to both the ergodicity of the Markov process as well as the
observability of the HMM.   
The proofs are based upon a recently discovered minimum variance
duality which is used to transform the nonlinear filtering problem
into a  stochastic optimal control problem for a backward stochastic
differential equation (BSDE).  
\end{abstract}

\begin{IEEEkeywords}
	Nonlinear filtering; Optimal control; Stochastic systems.
\end{IEEEkeywords}

\section{Introduction}\label{sec:intro}

This paper is on the topic of nonlinear (stochastic) filter
stability -- in the sense
of asymptotic forgetting of the initial condition.  The results are described for the continuous-time 
hidden Markov model (HMM) with white noise observations.  The
novelty comes from the methodological aspects which here are based on
the minimum variance 
duality introduced in our prior work: dual
characterization of stochastic observability presented in~\cite{duality_jrnl_paper_I}; and the dual
optimal control problem described in~\cite{duality_jrnl_paper_II}.  In
the present paper, these are used to investigate the question of
nonlinear filter stability,

\subsection{Literature review of filter stability}

While duality is central to the stability analysis of the Kalman
filter and also in the study of deterministic minimum energy estimator (MEE)~\cite{rawlings2017model},  with the
sole exception of van Handel's PhD thesis~\cite{van2006filtering},
duality is absent in stochastic filter stability theory.  
Viewed from a certain lens, the story of filter stability
is a story of two parts: (i) Stability of the Kalman
filter where dual (control-theoretic)
definitions and methods are paramount; and 
(ii) Stability of the
nonlinear filter where there is little hint of such methods. 

The disconnect is already seen in the earliest works -- in the two parts
of the pioneering 1996 paper of Ocone and Pardoux~\cite{ocone1996asymptotic} on the topic of
filter stability.
The paper is divided into two parts: Sec.~2 of the
paper considers the linear Gaussian model and the Sec.~3 considers
the nonlinear models (HMM).
While the problem is the same, the definitions, tools and techniques
of the two sections have no overlap. In~\cite[Sec.~2]{ocone1996asymptotic},
there are several control-theoretic definitions given, optimal control
techniques employed for analysis, references cited, while in~\cite[Sec.~3]{ocone1996asymptotic}
there are none. The passage of time did not change matters much: In
his award winning 2010 MTNS review
paper, van Handel writes
``{\it The proofs of the Kalman filter results are of essentially
	no use [for nonlinear filter stability], so we
	must start from scratch}~\cite{van2010nonlinear}.''

Our paper is the first
time that such a complete generalization of the linear Gaussian results has been
possible based on the use of duality.  A summary of the two main
contributions is described as part
of~\Sec{sec:original_contributaions} after the literature review.

The problem of nonlinear filter stability is far from
straightforward. In fact,~\cite[Sec.~3]{ocone1996asymptotic} is based
on some prior work of Kunita~\cite{kunita1971asymptotic} which was later found to contain a
gap, as discussed in some detail in~\cite{baxendale2004asymptotic}
(see also~\cite[Sec.~6.2]{JinPhDthesis}).
The gap also served to invalidate the main result in~\cite[Sec.~3]{ocone1996asymptotic}.
The literature on filter stability is divided into two cases: 
\begin{itemize}
\item The case where the Markov process forgets
the prior and therefore the filter ``inherits'' the same property; 
\item The case where the observation
provides sufficient information about the hidden state, allowing the
filter to correct its erroneous initialization.
\end{itemize}
These two cases are referred to as the ergodic and non-ergodic signal
cases, respectively.  While the two cases are intuitively reasonable,
they spurred much work during 1990-2010 with a complete
resolution appearing only at the end of this
time-period. See~\cite{chigansky2006stability,chigansky2009intrinsic}
for a comprehensive survey of the filter stability problem including
some of this historical context.

For the ergodic
signal case, apart from the pioneering
contribution~\cite{ocone1996asymptotic}, early work is based
on contraction analysis of the random matrix products arising from
recursive application of the Bayes' formula~\cite{atar1997lyapunov} (see also~\cite[Ch. 4.3]{Moulines2006inference}).
The analysis of the Duncan-Mortensen-Zakai (DMZ) equation leads to useful 
formulae for the Lyapunov exponents under assumptions on model
parameters and noise limits~\cite{atar1997exponential}, and
convergence rate estimates are obtained using Feynman-Kac type
representation~\cite{Atar1999}. A comprehensive account for the ergodic signal case appears
in~\cite{budhiraja2003asymptotic} and the first complete solution
appeared in~\cite{baxendale2004asymptotic}.   

For the non-ergodic signal case, a notable early
contribution is~\cite{clark1999relative} where a formula for the
relative entropy is derived.  It is shown that the relative entropy
is a Lyapunov function for the filter (see Rem.~\ref{rem:KL_divergence}).  Notable also is the
partial differential equation (PDE) approach
of~\cite{stannat2005stability,stannat2006} where sufficient conditions
for filter stability are described for a certain class of HMMs on the
Euclidean state-space with
linear observations (see
also~\cite[Ch.~4]{van2006filtering}).  
Our own prior work~\cite{kim2021detectable,kim2019observability} is closely inspired by~\cite{baxendale2004asymptotic} who
were the first to formulate certain observability-type ``identifying
conditions'' for the HMM on finite state-space.  These
conditions were formulated in terms of the HMM model parameters and
shown to be sufficient for the stability of the Wonham filter.

For a general class of HMMs, the fundamental definition for stochastic
observability and detectability is due to van Handel~\cite{van2009observability,van2009uniform}.  There
are two notable features: (i) the definition made
rigorous the intuition described in the two
cases~\cite[Sec. II-B and Sec. V]{van2010nonlinear}; and (ii) the
definition led to meaningful conditions that were shown to be 
necessary and sufficient for filter stability~\cite[Thm.~III.3 and
Thm.~V.2]{van2010nonlinear}.  The proof techniques are broadly
referred to as the {\em intrinsic approach}.
In~\cite{chigansky2009intrinsic}, the authors explain 
	``{\it By `intrinsic' we mean methods which directly exploit
          the fundamental representation of the filter as a
          conditional expectation through classical probabilistic
          techniques.}''  Recent
        extensions and refinements of these can be found
        in~\cite{mcdonald2018stability,mcdonald2019cdc,mcdonald2022robustness}.

A thorough mathematical review of these past approaches to the problem of
filter
stability appears in the PhD thesis of the first
author~\cite[Ch.~6]{JinPhDthesis}.  Additional mathematical 
comparisons with the approach of the present paper appear as part 
of Remarks~\ref{rm:contraction-literature}-\ref{rem:KL_divergence}
in~\Sec{ssec:literature-comparison} and Table~\ref{tb:rates} in
Appendix~\ref{apdx:forward-map-rates}.

\subsection{Summary of original contributions}\label{sec:original_contributaions}

The two main contributions are as follows:
\begin{enumerate}
\item The paper introduces a new notion of {\em Poincar\'e inequality}
  (PI) for the nonlinear filter.  The PI is used to obtain a novel
  formula~\eqref{eq:exp_stability_formula} for the filter stability. 
\item PI is related
to the two model properties, namely, the observability of the HMM, and the ergodicity of the signal
model (\Prop{prop:sufficiency}).  
\end{enumerate}

A key contribution is the {\em variance
  decay property} (Eq.~\eqref{eq:VDP}).  The property at 
once unifies and generalizes two bodies of results where the notion is variance is important:
\begin{itemize}
	\item Stability analysis of the Kalman filter:  The notion of
          variance is the conditional covariance (also referred to as the error covariance), which recall is given by
          the solution of the DRE.  
	\item Stochastic stability: The notion of
          variance is related to the Poincar\'e inequality (PI) which is a standard assumption to
          conclude asymptotic
          stability of a Markov process (without conditioning). 
          \end{itemize}

\subsection{Outline}

The outline of the remainder of this paper is as follows. The
mathematical background on HMMs and the filter stability problem
appears in~\Sec{sec:problem-formulation}.  The two central concepts in
this paper -- the backward map and the variance decay property --
are introduced in~\Sec{sec:backward_map}.  \Sec{sec:func_notn_assmp}
contains a discussion of function
spaces and the dual optimal control problem.  This is followed by
two sections that describes the two main contributions:
\Sec{sec:dual-optimal-ctrl} introduces the Poincar\'e inequality for nonlinear filter and~\Sec{sec:model_prop} describes its relationship to the HMM model properties for a finite state-space model.    
The paper closes with some conclusions and
directions for future work in~\Sec{sec:conc}. 
Details of the proofs appear in the Appendix.

\medskip

\section{Math Preliminaries and problem statement}\label{sec:problem-formulation}

\subsection{Hidden Markov model}

\newP{HMM}
On the probability space $(\Omega, \clF_T, \sP)$, consider a pair
of continuous-time stochastic processes $(X,Z)$ as follows:
\begin{itemize}
	\item The \emph{state process} $X = \{X_t:\Omega\to \bS:0\le t \le
	T\}$ is a Feller-Markov process taking values in the state-space
	$\bS$ which is assumed to be a locally compact Polish
        space. The prior is denoted by $\mu \in \clP(\bS)$ (where
        $\clP(\bS)$ is the space of probability measures defined on
        the Borel $\sigma$-algebra on $\bS$) and $X_0\sim \mu$. The infinitesimal generator of $X$ is
	denoted by $\clA$. 
	
	\item  The \emph{observation process} $Z = \{Z_t:0\le t \le T\}$ satisfies the stochastic differential
	equation (SDE):
	\begin{equation}\label{eq:obs-model}
		Z_t = \int_0^t h(X_s) \ud s + W_t,\quad t \ge 0
	\end{equation}
	where $h:\bS\to \Re^m$ is referred to as the observation function and $W =
	\{W_t:0\le t \le T\}$ is an $m$-dimensional Brownian motion
	(B.M.). We write $W$ is $\sP$-B.M.
	It is assumed that $W$ is independent of $X$. The filtration generated by the observation is denoted by
	$\clZ :=\{\clZ_t:0\le t\le T\}$  where $\clZ_t = \sigma\big(\{Z_s:0\le
	s\le t\}\big)$.
\end{itemize}
The above is referred to as the \emph{white noise observation model} of 
nonlinear filtering. The model is denoted by $(\clA,h)$.  For reasons
of well-posedness, the model requires additional technical
conditions.  In lieu of stating these conditions for general
class of HMMs, we restrict our study to the examples described in the following:

\begin{example}[Examples of state processes]\label{ex:state_processes}
  The two examples are as follows:
\begin{itemize}
	\item $\bS = \{1,2,\ldots,d\}$.  A real-valued function $f$ is
	identified with a vector in $\Re^d$ where the $x-{\text{th}}$
        element of the vector is $f(x)$ for $x\in\bS$.  Based on this, the
observation function $h$ is a $d\times m$ matrix and
the generator $\clA$ is a $d\times d$ transition rate matrix,
whose $(x,y)$ entry (for $x,y\in\bS$ and $x\neq y$) is the positive rate of transition from
$x\mapsto y$ and $\clA(x,x)=-\sum_{y:y\neq x} \clA(x,y)$.  
	\item $\bS\subseteq\Re^d$. $X$ is an It\^{o} diffusion process defined by:
	\begin{equation}\label{eq:ito_diffusion}
	\ud X_t = a(X_t)\ud t + \sigma(X_t)\ud B_t,\quad X_0\sim \mu
	\end{equation}
	where $a(\cdot)$ and $\sigma(\cdot)$ are given $C^1$ smooth
        functions that satisfy linear growth conditions at $\infty$.  The infinitesimal generator $\clA$ is given by~\cite[Thm. 7.3.3]{oksendal2003stochastic}
	\[
	(\clA f)(x) = a^\tp(x) \nabla f(x) + \half \tr\big(\sigma\sigma^\tp(x)(D^2f)(x)\big),\;x\in\bS
	\]
	where $\nabla f$ and $D^2f$ are the gradient vector and the
        Hessian matrix, respectively, of the function $f\in C^2(\Re^d)$.
	\item The linear Gaussian model is the special case of an
          It\^{o} diffusion where the drift $a(\cdot)$ is linear,
          $\sigma$ is a constant matrix, and the prior $\mu$ is a
          Gaussian density. 
\end{itemize}
\end{example}

\medskip

\begin{remark}
Of the two models of state processes, the HMM on a finite state-space is of the most interest.
We continue to use the notation $(\clA,h)$ and state the results in their
general form with the understanding that, for the general class of
HMMs, the calculations are formal.
\end{remark}

\newP{Nonlinear filter}
The objective of nonlinear (or stochastic) filtering is to compute the conditional expectation
\[
\pi_T(f):= \E\big(f(X_T)\mid \clZ_T\big),\quad f \in C_b(\bS)
\]
where $ C_b(\bS)$ is the space of continuous and bounded functions. 
The conditional expectation is referred to as the {\em nonlinear filter}. 
Assuming a certain technical (Novikov's) condition holds, the
nonlinear filter solves the \emph{Kushner-Stratonovich equation}~\cite{xiong2008introduction}: 
\begin{equation}\label{eq:Kushner}
	\ud \pi_t(f) = \pi_t(\clA f) \ud t +
        \big(\pi_t(hf)-\pi_t(f)\pi_t(h)\big)^\tp \ud I_t
\end{equation}
with $\pi_0=\mu$ where the \emph{innovation process} is defined by
\[
I_t := Z_t - \int_0^t \pi_s(h)\ud s,\quad t \ge 0
\]
With $h=c\ones$, the coefficient of $\ud I_t$ in~\eqref{eq:Kushner} is
zero and $\{\pi_t:t\geq 0\}$ becomes a deterministic process.  The
resulting evolution equation is referred to as the forward Kolmogorov equation.

\subsection{Filter stability: Definitions and metrics}\label{sec:filter_stab_defn}

Let $\rho \in \clP(\bS)$.  On the common measurable space $(\Omega,
\clF_T)$, $\sP^\rho$ is used to denote another probability measure
such that the transition law of $(X,Z)$ is identical but $X_0\sim
\rho$ (see~\cite[Sec.~2.2]{clark1999relative} for an explicit construction
of $\sP^\rho$ as a probability measure over
the space of the trajectories of the process $(X,Z)$.).  The associated expectation operator is denoted by
$\E^\rho(\cdot)$ and the nonlinear filter by $\pi_t^\rho(f) =
\E^\rho\big(f(X_t)|\clZ_t\big)$.  It
solves~\eqref{eq:Kushner} with $\pi_0=\rho$.  The
two most important choices for $\rho$ are as follows:
\begin{itemize}
\item $\rho=\mu$.
 The measure $\mu$ has the meaning of the true prior.    
\item  $\rho=\nu$. 
The measure $\nu$ has
  the meaning of the incorrect prior that is used to compute the
filter by solving~\eqref{eq:Kushner} with $\pi_0=\nu$.  It is assumed
that $\mu\ll\nu$.
\end{itemize}
The relationship between $\sP^\mu$ and $\sP^\nu$ is as follows ($\sP^\mu|_{\clZ_t}$ denotes
the restriction of $\sP^\mu$ to the $\sigma$-algebra $\clZ_t$):

\begin{lemma}[Lemma 2.1 in \cite{clark1999relative}] \label{lm:change-of-Pmu-Pnu}
	Suppose $\mu\ll \nu$. Then 
	\begin{itemize}
		\item $\sP^\mu\ll\sP^\nu$, and the change of measure is given by
		\begin{equation*}\label{eq:P-mu-P-nu}
			\frac{\ud \sP^\mu}{\ud \sP^\nu}(\omega) =
                        \frac{\ud \mu}{\ud
                          \nu}\big(X_0(\omega)\big)\quad
                        \sP^\nu\text{-a.s.} \;\omega
		\end{equation*}
		\item For each $t > 0$, $\pi_t^\mu \ll \pi_t^\nu$, $\sP^\mu|_{\clZ_t}$-a.s..
	\end{itemize} 
	
\end{lemma}

The following definition of filter stability is based on
$f$-divergence (Because $\mu$ has the meaning of the correct prior,
the expectation is with respect to $\sP^\mu$):

\begin{definition}\label{def:filter-stability}
	The nonlinear filter is \emph{stable} in the sense of 
	\begin{align*}
		\text{(KL divergence)}\qquad& \E^\mu\big(\kl(\pi_T^\mu\mid \pi_T^\nu)\big) \; \longrightarrow\; 0\\
		\text{($\chi^2$ divergence)}\qquad& \E^\mu\big(\chisq(\pi_T^\mu\mid \pi_T^\nu)\big) \; \longrightarrow\; 0\\
		\text{(Total variation)}\qquad& \E^\mu\big(
                                                \|\pi_T^\mu - \pi_T^\nu\|_\tv\big) \; \longrightarrow\; 0
	\end{align*}
	as $T\to \infty$ for every $\mu, \nu\in\clP(\bS)$ such that
        $\mu\ll \nu$. (See Appendix~\ref{ss:pf-prop71} for definitions
        of the $f$-divergence).
	
\end{definition}

Apart from $f$-divergence based definitions, the following definitions of 
filter stability are also of historical interest:

\begin{definition}\label{def:FS}
	The nonlinear filter is stable in the sense of
	\begin{align*}
		\text{($L^2$)}\qquad & \E^\mu\big(|\pi_T^\mu(f)-\pi_T^\nu(f)|^2\big) \;\longrightarrow\; 0\\
		\text{(a.s.)}\qquad & |\pi_T^\mu(f) - \pi_T^\nu(f)|\;\longrightarrow\; 0\quad \sP^\mu\text{-a.s.}
	\end{align*}
	as $T\to \infty$, for every $f\in C_b(\bS)$ and $\mu,\nu\in\clP(\bS)$ s.t.~$\mu \ll \nu$.
	
\end{definition}

In this paper, our objective is to prove filter stability in the
sense of $\chisq$-divergence.  Based on well known relationship between
f-divergences, this also implies other types of  stability as follows:

\begin{proposition}\label{prop:chisq-stability-implication}
	If the filter is stable in the sense of $\chisq$ then it is stable in KL divergence, total variation, and $L^2$.
\end{proposition}

\begin{proof}  See Appendix~\ref{ss:pf-prop71}. 
\end{proof}

Because these were stated piecemeal, the main assumptions 
are stated formally as follows:

\newP{Assumption 0} Consider HMM $(\clA,h)$.  
\begin{enumerate}
\item $\mu,\nu\in\clP(\bS)$ are two priors with $\mu\ll\nu$. 
\item Novikov's condition holds:
\[
\E \left(\exp\big(\half \int_0^\tau |h(X_t)|^2\ud t\big)\right) < \infty
\] 
The condition holds, e.g., if $h\in C_b(\bS)$.  
\item The generator $\clA$ is for one of the two models introduced in Example~\ref{ex:state_processes}.
\end{enumerate}

\section{Main idea: Backward map and variance decay}\label{sec:backward_map}

Suppose $\mu\ll\nu$.  Denote 
\[
\gamma_T(x):=\frac{\ud \pi_T^\mu}{\ud
	\pi_T^\nu}(x),\quad x\in\bS
\]
It is well-defined because $\pi_T^\mu \ll \pi_T^\nu$
from Lemma~\ref{lm:change-of-Pmu-Pnu} (we adopt here the convention that $\frac{0}{0}=0$).  It is noted that while
$\gamma_0= \frac{\ud 
		\mu}{\ud \nu}$ is deterministic, $\gamma_T$ is a
$\clZ_T$-measurable function on $\bS$.  Both of these are examples of 
  likelihood ratio and referred to as such throughout the paper. 

A key original concept introduced in this paper is the \emph{backward map}
$\gamma_T\mapsto y_0$ defined as follows:
\begin{equation}\label{eq:bmap}
	y_{0}(x) :=  \E^\nu (\gamma_T(X_T)|[X_0=x]),\quad x\in\bS
\end{equation}
The function $y_0:\bS\to \Re$ is deterministic, non-negative, and 
$\nu(y_0) =  \E^\nu (\gamma_T(X_T)) = 1$, and therefore is also a
likelihood ratio. 
The significance of this map to the problem of filter stability comes from the 
following proposition:

\begin{proposition} \label{prop:var-decay-implies-filter-stability}
Consider the backward map $\gamma_T\mapsto y_0$ defined
by~\eqref{eq:bmap}.  Then
\begin{equation}\label{eq:chisq_identity_intro}
	|\E^\mu\big(\chi^2(\pi_T^\mu\mid\pi_T^\nu)\big)|^2\leq
\var^\nu (y_0(X_0)) \; \chisq(\mu | \nu)
\end{equation}
where $\var^\nu(y_0(X_0))=\E^\nu\big( |y_0(X_0)-1|^2\big)$. 
\end{proposition}

\begin{proof}
Since $\mu\ll\nu$, it follows $\mu(y_0) = \E^\mu\big(\gamma_T(X_T)\big)$. Using the tower property, 
\[
\mu(y_0) = \E^\mu\big(\gamma_T(X_T)\big) = \E^\mu\big(\pi_T^\mu(\gamma_T)\big) = \E^\mu\big(\pi_T^\nu(\gamma_T^2)\big)
\]
Noting that $\pi_T^\nu(\gamma_T^2)-1 = \chisq\big(\pi_T^\mu\mid\pi_T^\nu\big)$ is the $\chisq$-divergence, 
\[
\E^\mu\big(\chisq\big(\pi_T^\mu\mid\pi_T^\nu\big)\big) = \mu(y_0) -
\nu(y_0)
\]
Because 
$
\mu(y_0)-\nu(y_0) = \nu\big((\gamma_0-1)(y_0-1)\big)
$, 
upon using the Cauchy-Schwarz inequality
gives~\eqref{eq:chisq_identity_intro}. 
\end{proof}

From~\eqref{eq:chisq_identity_intro}, provided $\chisq(\mu\mid\nu)
<\infty$, 
a sufficient condition for filter stability is the following:
\begin{flalign}\label{eq:VDP}
	&\textbf{(variance decay prop.)}\quad  \var^\nu\big(y_0(X_0)\big)  \stackrel{(T\to\infty)}{\longrightarrow}
	0 &
\end{flalign}
Next, from~\eqref{eq:bmap}, $
(y_0(X_0) -1) =  \E^\nu \big((\gamma_T(X_T)-1)|X_0\big)$, 
and using
Jensen's inequality, 
\begin{equation}\label{eq:jensens_ineq}
	\var^\nu\big(y_0(X_0)\big) \le \var^\nu\big(\gamma_T(X_T)\big)
\end{equation}
where $\var^\nu(\gamma_T(X_T)):=\E^\nu\big(|\gamma_T(X_T)-1|^2\big)$. 
Therefore, the backward map $\gamma_T\mapsto y_0$ is non-expansive -- the
variance of the random variable $y_0(X_0)$ is smaller than the
variance of the random variable $\gamma_T(X_T)$.

In the remainder of this paper, we have two goals:
\begin{enumerate}
\item To express a stronger form of~\eqref{eq:jensens_ineq} such that
  the variance decay property~\eqref{eq:VDP} is deduced under a
  suitable definition of the Poincar\'e inequality (PI).
\item Relate PI to the model properties, namely,
  (i) ergodicity of the Markov process; and (ii) 
  observability of the HMM~$(\clA,h)$.
\end{enumerate}
Concerning these goals, the contributions of this paper are noted briefly as an aid to the reader.  These are as follows:
\begin{enumerate}
\item The stronger form of~\eqref{eq:jensens_ineq} is formula~\eqref{eq:var_contractive}.
\item Relationship of the PI to the HMM model properties is given in~\Prop{prop:sufficiency}.
\item Based on the use of the PI, the two main filter stability results are given in~\Thm{thm:general-nonexponential-case} and~\Thm{thm:finite-case}. 
\end{enumerate}

The following subsections are included to help relate the approach of this
paper to the
literature.  The reader may choose to skip ahead to
\Sec{sec:func_notn_assmp} without any loss of continuity. 

\subsection{Comparison to literature} \label{ssec:literature-comparison}

\begin{remark}[Contraction analysis] \label{rm:contraction-literature}
Based on~\eqref{eq:jensens_ineq}, the variance decay is a contraction property of the backward linear map $\gamma_T\mapsto
y_0$.  This nature of contraction analysis is contrasted with the contraction analysis of the random matrix
products arising from recursive application of the Bayes'
formula~\cite{atar1997lyapunov}~\cite[Ch. 4.3]{Moulines2006inference}. For
the HMM with white noise observations $(\clA,h)$,  the random linear
operator is the solution operator of the DMZ equation~\cite{atar1997exponential}.  An
early contribution on this theme appears in~\cite{delyon1988lyapunov},
which was expanded in~\cite{atar1997exponential,atar1997lyapunov,atar1998exponential}.
In these papers, the stability index is defined by
\[
\overline{\gamma}:=\limsup_{T\to\infty}\frac{1}{T}\log\|\pi_T^{\nu}-\pi_T^\mu\|_\tv
\]
If this value is negative then the filter is asymptotically stable in
total variation norm. Moreover, $-\overline{\gamma}$ represents the
rate of exponential convergence.  A summary of known bounds for
$\overline{\gamma}$ is given in Appendix~\ref{apdx:forward-map-rates} and compared to the
bounds obtained using the approach of this paper. 
\end{remark}

\begin{remark}[Forward map] \label{rm:forward-map} The backward
  map $\gamma_T\mapsto y_0$ is contrasted with the forward map
  $\gamma_0\mapsto \gamma_T$ defined as
  follows~\cite[Lemma 2.1]{clark1999relative}:
\begin{equation*}\label{eq:fmap}
	\gamma_T(x) = \E^\nu
	\Big( \frac{\gamma_0(X_0)}{\E^\nu(\gamma_0(X_0) \mid \clZ_T)}
        \,\bigg|\,\clZ_T\vee [X_T=x]\Big),\quad x\in \bS
\end{equation*}
The forward map is the starting point of the intrinsic (probabilistic)
approach to filter stability~\cite{chigansky2009intrinsic}.   
Both the forward and backward maps have as their domain and range the space of likelihood
ratios. 
While the
forward map is nonlinear and random, the  backward
map~\eqref{eq:bmap} is linear and deterministic.  

A marvelous
  success of the intrinsic approach is to establish filter
  stability in total
  variation for the ergodic signal
  case~\cite[Thm.~III.3]{van2010nonlinear} and a.s. for the
  observable case~\cite[Thm.~1]{van2009observability}.  
\end{remark}

\begin{remark}[Metrics for likelihood ratio] 
\label{rem:KL_divergence}
In an important early study, the following formula
for KL divergence (or relative entropy) is shown~\cite[Thm~2.2]{clark1999relative}:
\begin{align*}
	\kl(\mu\mid\nu) &\geq \E^\mu\big(\kl(\pi_t^\mu\mid\pi_t^\nu)\big) + 
	\kl\big(\sP^\mu|_{\clZ_t}\mid\sP^\nu|_{\clZ_t}\big),
\quad t> 0
\end{align*}
From this formula, a corollary is that
$\{\kl(\pi_t^\mu\mid\pi_t^\nu):t\geq 0\}$ is a
  non-negative $\sP^\mu$-super-martingale (assuming $\kl(\mu\mid\nu)<\infty$).  Therefore, the relative
  entropy is a Lyapunov function for the filter, in
  the sense that $\E^\mu\big(\kl(\pi_t^\mu\mid\pi_t^\nu)\big)$ is
  non-increasing as function of time.  However,
  it is difficult to establish conditions that show that~$\E^\mu\big(\kl(\pi_T^\mu\mid\pi_T^\nu)\big) \stackrel{(T\to\infty)}{\longrightarrow} 0$~\cite[Sec.~4.1]{chigansky2009intrinsic}.
For white noise
observations model~\eqref{eq:obs-model}, an explicit formula is obtained as
follows~\cite[Thm.~3.1]{clark1999relative}:
\[
\kl \big(\sP^\mu|_{\clZ_t}\mid\sP^\nu|_{\clZ_t}\big) = \half 
\E^\mu\Big(\int_0^t |\pi_s^\mu(h)-\pi_s^\nu(h)|^2 \ud s\Big)
\]
Therefore, $\E(|\pi_t^\mu(h)-\pi_t^\nu(h)|^2) \to 0$ which
shows that the
filter is \emph{always} stable for the observation function
$h(\cdot)$. A generalization is given in~\cite{chigansky2006role} where
it is proved that one-step predictive estimates of the observation
process are stable.  These early results served as the foundation for
the definition of stochastic observability introduced
in~\cite{van2009observability}.
\end{remark}

\subsection{Background  on PI for a Markov process}

To see the importance of PI in the study of Markov processes, let us consider
the $\chisq$-divergence with $h=c\ones$.  In this case,
the two processes $\{\pi_t^\mu:t\geq 0\}$ and $\{\pi_t^\nu:t\geq 0\}$
are both deterministic and a straightforward calculation (see
Appendix~\ref{apdx:chisq}) shows that
\begin{equation}\label{eq:rate-chisq-Markov}
\frac{\ud }{\ud t} \,\chisq(\pi_t^\mu\mid \pi_t^\nu)  = -\pi^\nu_t\big(\Gamma 
	\gamma_t\big)
\end{equation}
where $\Gamma$ is the so called carr\'e du champ operator.  Its formal
definition is as follows:

\begin{definition}[Defn.~1.4.1. in~\cite{bakry2013analysis}] The
  bilinear operator
\[
\Gamma(f, g)(x) := (\clA fg)(x) - f(x)(\clA g)(x) - g(x)(\clA f)(x),\;x\in\bS
\]
defined for every $(f,g)\in\clD \times \clD$ is called the carr\'e du
champ operator of the Markov generator $\clA$. Here, $\clD$ is a 
vector space of (test) functions that are dense in $L^2$, stable under
products (i.e., $\clD$ is an algebra), and $\Gamma:\clD\times
\clD\to \clD$ (i.e., $\Gamma$ maps two functions in $\clD$ into a
function in $\clD$), such that $\Gamma(f,f)\geq 0$ for every
$f\in\clD$~\cite[Defn.~3.1.1]{bakry2013analysis}. 
$(\Gamma f)(x):= \Gamma(f,f)(x)$.   
\end{definition}

\begin{example}[Continued from Ex.~\ref{ex:state_processes}] \label{ex:carre-du-champ}For the
  examples of the state processes in Ex.~\ref{ex:state_processes}],
  the carr\'e du champ operators are as follows:
\begin{itemize}
	\item $\bS = \{1,2,\ldots,d\}$. Then
         \begin{equation*}\label{eq:Gamma-finite}
 		(\Gamma f)(x) = \sum_{y \in \mathbb{S}} \clA(x,y) (f(x) - f(y))^2,\quad x\in\bS
 	\end{equation*}
for $f\in\clD=\Re^d$. The same definition also applies to discrete
state-spaces with countable cardinality. 
	\item $\bS=\Re^d$. For the It\^o diffusion~\eqref{eq:ito_diffusion}, the carr\'e du champ operator is given by 
	\begin{equation*}\label{eq:Gamma-Euclidean}
		(\Gamma f) (x) = \big|\sigma^\tp(x) \nabla f(x) \big|^2,\quad x\in\Re^d
	\end{equation*}
	for $f\in \clD=C^1(\Re^d;\Re)$.
\end{itemize}
\end{example}

Returning to~\eqref{eq:rate-chisq-Markov}, an important point to note
is that $\Gamma$ is positive-definite and thus the right-hand
side of~\eqref{eq:rate-chisq-Markov} is non-positive.  
This means $\chisq$-divergence is a candidate Lyapunov function.
To show $\chisq$-divergence asymptotically goes to zero requires
additional assumption on the model.  PI is one such assumption.  It is described
next.  

Suppose $\bmu\in\clP(\bS)$ is an invariant probability measure and
let $L^2(\bmu):=\{f:\bS\to \Re \mid \bmu(f^2) < \infty\}$.  
The Poincar\'e constant is defined as follows:
\[
c := \inf\{\bmu(\Gamma f) \; :\; f\in L^2(\bmu) \; \& \; \dvar^\bmu(f(X_0)) = 1\}
\]
When the Poincar\'e constant $c$ is strictly positive the resulting inequality is referred to as the
Poincar\'e inequality (PI):
	\begin{equation*}
		\text{(PI)}\qquad\qquad	\bmu(\Gamma f) \geq c \;
		\dvar^\bmu(f(X_0))\quad \forall\, f \in L^2(\bmu)
	\end{equation*}
The significance of the PI to the problem at hand is as follows: Set $\nu =
\bmu$.  Then $\gamma_t = \frac{\pi_t^\mu}{\bmu}$ and the differential equation~\eqref{eq:rate-chisq-Markov} for $\chisq$-divergence becomes 
	\begin{align*}
	\frac{\ud}{\ud t} \chisq(\pi_t^\mu\mid \bmu) & = - \bmu
        (\Gamma \gamma_t) \\ & \stackrel{\text{(PI)}}{\leq} -c \: \dvar^\bmu(\gamma_t
        (X_0))  = - c \: \chisq(\pi_t^\mu\mid\bmu) 
	\end{align*}
Therefore, provided $\chisq(\mu \mid\bmu)< \infty$, asymptotic
stability in the sense of $\chi^2$-divergence is shown (The
Poincar\'e constant $c$ gives the exponential rate of decay).

\begin{remark}
PI provides a natural definition for ergodicity of a continuous-time
Markov process.  The relationship between PI and the Lyapunov approach
of Meyn-Tweedie is described at length in~\cite{bakry2008rate}.
Specifically, it is shown that (i) existence of a positive Poincar\'e
constant is equivalent to exponential stability (in the sense of
$\E(f(X_t)) \to \bmu(f)$ for $f\in L^2(\bmu)$), and (ii) existence of
a Lyapunov function from Meyn-Tweedie theory implies a positive Poincar\'e constant~\cite[Thm.~4.6.2]{bakry2013analysis}. 
\end{remark}

A goal in this paper is to define an
appropriate notion of the PI for
the HMM $(\clA,h)$ and use it to show filter stability.

\section{Function spaces, notation, and duality}\label{sec:func_notn_assmp}

\subsection{Function spaces}

Let $\rho\in\clP(\bS)$ and $\tau>0$.  These are used to denote a
generic prior and a generic time-horizon $[0,\tau]$. (In the analysis
of filter stability, these are fixed to $\rho=\nu$ and $\tau=T$).  The
space of Borel-measurable deterministic functions is denoted
\[
L^2(\rho)=\{f:\bS\to \Re \;:\;\rho(f^2) = \int_\bS |f(x)|^2 \ud\rho (x) <\infty\}
\]

\newP{Background from nonlinear filtering} 
A standard approach 
is based upon the Girsanov change of measure. Because the Novikov's
condition holds, define a new measure $\tsP^\rho$
on $(\Omega,\clF_\tau)$ as follows:
\[
\frac{\ud \tsP^\rho}{\ud \sP^\rho} = \exp\Big(-\int_0^\tau
h^\tp(X_t) \ud W_t - \half \int_0^\tau |h(X_t)|^2\ud t\Big) =: D_\tau^{-1}
\]
Then the probability law for $X$ is unchanged but
$Z$ is a $\tsP^\rho$-B.M.~that is independent of $X$~\cite[Lem.~1.1.5]{van2006filtering}. The
expectation with respect to $\tsP^\rho$ is denoted by
$\tE^\rho(\cdot)$.  The unnormalized filter $\sigma_\tau^\rho(f) :=
\tE^\rho(D_\tau f(X_\tau)|\clZ_\tau)$ for $f\in C_b(\bS)$.  It is
called as such because $\pi_\tau^\rho(f) = \frac{\sigma_\tau^\rho(f)}{\sigma_\tau^\rho(\ones)}$.  The
measure-valued process $\{\sigma_t^\rho:0\leq t\leq \tau\}$ is the solution of
the DMZ equation.

\begin{table}[t]
	\centering
	\renewcommand{\arraystretch}{2.0}
	\caption{Hilbert space for $\Re^m$-valued signals.} \label{tb:signal_inner_product}
	\small
	\begin{tabular}{p{0.1\textwidth}p{0.2\textwidth}}
        Notation & Inner-product \\ 
\hline \hline
$\clU$ & $\langle U, \tilde{U} \rangle := \tE^\rho ( \int_0^\tau
         U_t^\tp \tilde{U}_t \ud t )$ \\ 
\hline
        \end{tabular}
\end{table}

There are two types of function spaces:

\newP{$\bullet$ Hilbert space for signal} $\clU$ is used to denote the
  Hilbert space of $\Re^m$-valued $\clZ$-adapted stochastic processes.
  It is defined as $\clU:=
L^2\big(\Omega\times[0,\tau];\clZ\otimes \clB([0,\tau]);\ud \tsP^\rho\ud
t\big)$ 
where $\clB([0,\tau])$ is the Borel sigma-algebra on $[0,\tau]$,
$\clZ\otimes \clB([0,\tau])$ is the product sigma-algebra and $\ud
\tsP^\rho\ud t$ denotes the product measure on
it~\cite[Ch.~5.1.1]{le2016brownian}. See
Table~\ref{tb:signal_inner_product} for notation and definition of the 
inner product.

\begin{table*}[t]
	\caption{Function space for dual state (Left: $\bS=\{1,2,\hdots,d\}$, Right: $\bS\subseteq\Re^d $)} \label{tb:function_spaces}
	\renewcommand{\arraystretch}{2.0}
	\begin{minipage}{.45\linewidth}
		\centering
		\small
		\begin{tabular}{p{0.18\textwidth}p{0.67\textwidth}}
			Notation & Inner-product \\ 
			\hline \hline
			$\clY$ & $ \lambda(fg) := \sum_{x\in\bS} \lambda(x)
			f(x) g(x)$ 
			\\ \hline
			$\mathbb{H}_\tau^\rho $ & $ \langle F, G \rangle  := \tE^\rho
			(\sigma_{\tau}^\rho (FG)) $ \\
			& $\qquad\quad= \tE^\rho ( \sum_{x\in\bS}
			\sigma_{\tau}^\rho(x) F(x) G(x) )$ \\ \hline
			$\mathbb{H}^\rho([0,\tau])$ & $\langle Y, \tilde{Y} \rangle :=\tE^\rho \left( \int_0^\tau 
			\sigma_{t}^\rho (Y_t \tilde{Y}_t) \ud t \right)$ \\
			\hline
		\end{tabular}
	\end{minipage}%
\begin{minipage}{.55\linewidth}
	\centering
	\small
	\begin{tabular}{p{0.15\textwidth}p{0.7\textwidth}}
		Notation & Inner-product \\ 
		\hline \hline
		$\clY$ & $ \langle f, g \rangle_\lambda :=  \int_{\Re^d} ( f(x) g(x) +
		Df(x)^\tp Dg(x) \ud \lambda(x)$ 
		\\ \hline
		$\mathbb{H}_\tau^\rho $ & $ \langle F, G \rangle  := \tE^\rho
		(\langle F, G \rangle_{\sigma_{\tau}^\rho} ) $ \\
		& $ \;\; = \tE^\rho \left( \int_{\Re^d} ( F(x) G(x) +
		DF(x)^\tp DG(x) )\ud \sigma_{\tau}^\rho (x) \right)$ \\ \hline
		$\mathbb{H}^\rho([0,\tau])$ & $\langle Y, \tilde{Y} \rangle :=\tE^\rho \left( \int_0^\tau 
		\langle Y_t \tilde{Y}_t \rangle_{\sigma_{t}^\rho} \ud t \right)$ \\
		\hline
	\end{tabular}	
\end{minipage}%
\end{table*}

\newP{$\bullet$ Hilbert space for the dual}  Formally, the ``dual'' is a function on the state-space.  The space of such functions is denoted as
$\clY$.  It is easiest to
describe the Hilbert space first for the case when
$\bS=\{1,2,\hdots,d\}$.  In this case, $\clY=\Re^d$ (See
Ex.~\ref{ex:state_processes}).  Related to the dual, two types of Hilbert spaces are of interest.  These are defined
as follows:
\begin{itemize}
\item Hilbert space of $\clZ_\tau$-measurable random functions:
\begin{align*}
\mathbb{H}_\tau^\rho & :=\{ F:\Omega\to \clY \; : \; F\in\clZ_\tau \;\;
                       \& \;\; \tE^\rho
  (\sigma_{\tau}^\rho (F^2)) < \infty\} 
\end{align*}
(This function space is important because 
the backward map~\eqref{eq:bmap} 
is a map from
$\gamma_T \in \mathbb{H}_T^\nu$ to $y_0\in L^2(\nu)$). 
\item Hilbert space of $\clY$-valued $\clZ$-adapted stochastic processes:
\begin{align*}
\mathbb{H}^\rho([0,\tau]) & :=\{ Y:\Omega\times [0,\tau] \to \clY  :
                            \;  Y_t\in\clZ_t, \;0\leq t\leq \tau, \\
& \qquad\qquad 
                       \& \;\; \tE^\rho \left( \int_0^\tau 
  \sigma_{t}^\rho (Y_t^2) \ud t \right) < \infty\}
\end{align*}
(This function space is important because we will embed the backward
map~\eqref{eq:bmap} $\gamma_T \mapsto y_0$  into a $\clY$-valued $\clZ$-adapted stochastic process
$Y=\{Y_t:\Omega\to \clY :0\leq t\leq T\}$ 
such that $Y_T=\gamma_T$ and $Y_0=y_0$).  
\end{itemize}
An extension of these definitions to the case where $\bS\subseteq\Re^d$ is described in the
following example.

\begin{example}[Continued from Ex.~\ref{ex:state_processes}] For the
  examples of the state processes in Ex.~\ref{ex:state_processes}],
  the examples of $\clY$ are as follows:
\begin{itemize}
\item $\bS=\{1,2,\hdots,d\}$.  $\clY=\Re^d$ as discussed above.
\item $\bS\subseteq\Re^d$. $\clY=W^{1,2}(\Re^d)$ is a
  Sobolev space.  
\end{itemize}
For these two examples, the definition of the inner-products for
$\mathbb{H}_\tau^\rho$ and $\mathbb{H}^\rho([0,\tau])$ appear as part of Table~\ref{tb:function_spaces}. 
\end{example}

\subsection{Notation} 
Let $\rho\in{\cal P}(\bS)$.  
For real-valued functions $f,g\in \clY $, $
\clV_t^\rho(f,g) := \pi_t^\rho\big((f-\pi_t^\rho(f))(g-\pi_t^\rho(g))\big)
$. 
With $f=g$, $\clV_t^\rho(f) := \clV_t^\rho(f,f)$.
At time $t=0$,
$\clV_0^\rho(f) = \rho(f^2)-\rho(f)^2 = \E^\rho(|f(X_0)-\rho(f)|^2)=
\var^\rho(f(X_0))$.
In literature, $\clV_0^\rho(f)$ has been denoted as
``$\var^\rho(f)$'' and 
referred to as the ``variance of the function
$f$ with respect to $\rho$''~\cite[Eq.~(4.2.1)]{bakry2013analysis}.  In this paper, we will
instead adopt a more conventional terminology whereby the argument of
$\var^\rho(\cdot)$ is always a random variable.
Likewise, $\clV_t^\rho(f,g)$ is (related to) the conditional covariance
because $\clV_t^\rho(f,g)=\E^\rho((f(X_t)-\pi_t^\rho(f))
(g(X_t)-\pi_t^\rho(g))|\clZ_t)$, and $\clV_t^\rho(f)$ is the
conditional variance of $f(X_t)$.

Apart from real-valued functions, it is also necessary to consider
$\Re^m$-valued functions.  The space of such functions is denoted
$\clY^m$.  Let $v\in \clY^m$. For each $x\in\bS$, $v(x)$ is
a column vector $v(x)=[v^1(x),\ldots,v^m(x)]^\tp$ where $v^j \in
\clY$ for $j=1,2,\hdots,m$.   For $v,\tilde{v} \in \clY^m$, 
$\clV_t^\rho(v,\tilde{v}) := \pi_t^\rho\big((v-\pi_t^\rho(v))^\tp(\tilde{v}-\pi_t^\rho(\tilde{v}))\big)
$ and $\clV_t^\rho(v) := \clV_t^\rho(v,v) =
\pi_t^\rho\big(|v-\pi_t^\rho(v)|^2\big)$.  

For $f\in\clY$ and $v\in\clY^m$, $\clV_t^\rho(f,v) =
\pi_t^\rho \big( (f - \pi_t^\rho(f))(v-\pi_t^\rho(v)) \big):=
[\clV_t^\rho(f,v^1),\ldots,\clV_t^\rho(f,v^m)]^\tp$.

The space of likelihood
ratio is denoted by 
\begin{align*}
\clL^\rho_\tau:=\big\{F\in
\mathbb{H}_\tau^\rho : F(x)\geq 0, \;x\in\bS  \;\;\&\;\;
\pi^\rho_\tau(F) = 1,\;\sP^\rho\text{-a.s.} &\big\} 
\end{align*}

\subsection{Duality: Optimal control problem}

Our goal is to embed the backward
map~\eqref{eq:bmap} $\gamma_T \mapsto y_0$ into a continuous-time 
backward process.  For this purpose, the following dual optimal control problem
is considered.  The problem was
previously introduced by us in~\cite{duality_jrnl_paper_II}.
(Additional motivation is provided in Remark~\ref{rem:mot_OCP}).

\newP{Dual optimal control problem}
\begin{subequations}\label{eq:dual-optimal-control}
	\begin{align}
		&           \mathop{\text{Min:}}_{U\in\;\clU}\; \sJ_\tau^\rho(U)  =
		\var^\rho(Y_0(X_0)) + 
		\E^\rho \Big(\int_0^\tau l (Y_t,V_t,U_t\,;X_t) \ud t \Big)\label{eq:dual-optimal-control-a}\\
		&          \text{Subject to (BSDE constraint):} \nonumber \\ 
		&  -\!\ud Y_t(x) = \big((\clA Y_t)(x) + h^\tp (x) (U_t +
		V_t(x))\big)\ud t - V_t^\tp(x)\ud Z_t \nonumber \\
		&    \quad \;  Y_\tau (x)  = F(x), \;\; x \in \bS 
		\label{eq:dual-optimal-control-b}
	\end{align}
\end{subequations}
where $(Y,V) \in \mathbb{H}^\rho([0,\tau])\times \mathbb{H}^\rho([0,\tau])^m$ is the solution of~\eqref{eq:dual-optimal-control-b} for a
given $F\in \mathbb{H}_\tau^\rho $ and $U\in \clU$, 
and the running cost 
\begin{equation*}\label{eq:running_cost_formula}
	l(y,v,u;x)= (\Gamma y)(x) + |u+v(x)|^2
\end{equation*}
for $y\in\clY,v\in\clY^m,u\in\Re^m,x\in\bS$. (If $\bS$ is finite,
$\clY=\Re^d$).

The solution to~\eqref{eq:dual-optimal-control} and its relationship to the optimal filter is given in the following
theorem:

\begin{theorem}\label{thm:optimal-solution}
	Consider the optimal control problem~\eqref{eq:dual-optimal-control}.  
	The optimal control is of the feedback form given by
	\begin{equation}\label{eq:optimal_control_formula}
		U_t = U_t^\opt := -\clV_t^\rho(h,Y_t)-
		\pi_t^\rho(V_t),\;\;\sP^\rho\text{-a.s.},\quad 0\leq
		t\leq \tau
	\end{equation}
	Suppose $(Y,V)=\{(Y_t,V_t):0\leq t\leq \tau\}$ is the
	associated solution of the
	BSDE~\eqref{eq:dual-optimal-control-b}.  Then
	
	\noindent $\bullet$ For almost every
	$0\le t \le \tau$,
	\begin{subequations}
		\begin{align}
			&	\pi_t^\rho (Y_t) = \rho (Y_0) - \int_0^t \big(U_s^\opt \big)^\tp\ud Z_s,\quad \sP^\rho\text{-a.s.} \label{eq:estimator-t} \\
			&	\E^\rho\big(\clV_t^\rho (Y_t)\big) = 
			\clV_0^\rho (Y_0)
			+ \E^\rho\Big(\int_0^t l(Y_s,V_s,U_s^\opt;X_s)\ud s\Big)  \label{eq:estimator-t-variance}
		\end{align}
	\end{subequations}
		
		\noindent $\bullet$ Define a real-valued 
		$\clZ$-adapted process $M:=\{M_t:0\leq t\leq \tau\}$ as follows:
		\begin{equation}\label{eq:martingale}
			M_t := \clV_t^\rho (Y_t) - \int_0^t
			\E^\rho \big(\ell (Y_s,V_s,U_s^{\text{opt}}\,;X_s)|\clZ_s\big) \ud s
		\end{equation}
		Then $M$ is a $\sP^{\rho}$-martingale.
		
		\noindent $\bullet$ For $f\in\clY$,
		\begin{align}
			&	\ud \clV_t^\rho(f,Y_t)
			=\Big(\pi_t^\rho\big(\Gamma(f,Y_t)\big) + \clV_t^\rho(\clA f,Y_t)
			\Big)\ud t \nonumber\\
			& \quad + \Big(\clV_t^\rho\big( (f-\pi_t^\rho(f))(h-\pi_t^\rho(h)),Y_t\big)  + \clV_t^\rho(f,V_t) \Big)^\tp \ud
			I_t^\rho \label{eq:variance_Yt_f}
		\end{align}
	where $I_t^\rho := Z_t - \int_0^t \pi_s^\rho(h) \ud s$.
\end{theorem}

\begin{proof}
	The feedback control formula is given
	in~\cite[Thm.~3]{duality_jrnl_paper_II}. The equations for conditional
	mean and variance are in~\cite[Prop.~1]{duality_jrnl_paper_II}.
	The martingale characterization appears in~\cite[Thm.~3]{duality_jrnl_paper_II}.  
	In
	the form presented here, the SDE~\eqref{eq:variance_Yt_f} for the
	conditional variance is new (the form is used in the proof of the
	main result).  It is easily derived from the Hamilton's
	equation~\cite[Thm.~2]{duality_jrnl_paper_II} for the optimal control
	problem~\eqref{eq:dual-optimal-control}.
	The derivation is included in Appendix~\ref{apdx:pf-thm1}.  
\end{proof}

\begin{example}[Continued from Ex.~\ref{ex:state_processes}] 
There is a well developed theory for existence, uniqueness and
regularity of the solutions of the BSDE~\eqref{eq:dual-optimal-control-b}.  For the two state processes
of interest, the theory can be found in the following:
\begin{itemize}
\item $\bS=\{1,2,\hdots,d\}$ and $\clY=\Re^d$. See~\cite[Ch.~7]{yong1999stochastic}.
\item $\bS=\Re^d$ and
  $\clY=W^{1,2}(\Re^d)$. See~\cite[Thm.~3.2.]{ma1997adapted} where
  additional assumptions on the model are stated for these results to hold.
\end{itemize}
\end{example}

\begin{remark}\label{rem:mot_OCP}
	The optimal control problem~\eqref{eq:dual-optimal-control} is a
	generalization of the classical minimum variance duality to the HMM
	$(\clA,h)$ (See~\cite{duality_jrnl_paper_II} where historical context
	is provided).
	Formula~\eqref{eq:estimator-t}  gives the filter in terms of the
	solution of this problem.  The idea of this paper is to obtain
	conclusions on the asymptotic stability of the filter based on
	analysis of the optimal control system.  
	This is possible because of the relationship of the optimal control system to the backward map~\eqref{eq:bmap} as described next.
\end{remark}

\subsection{Relationship to the backward map~\eqref{eq:bmap}}

Recall the backward map~\eqref{eq:bmap} $\gamma_T\mapsto
y_0$ introduced in \Sec{sec:backward_map}.  The following relates it to the optimal control system.

\begin{proposition}\label{prop:backward-map-and-bsde}
	Consider the optimal control
	problem~\eqref{eq:dual-optimal-control} with $\rho=\nu$,
	$\tau=T$, and the terminal condition $Y_T=F=\gamma_T$.  Then at
	time $t=0$,
	\begin{equation*}\label{eq:bmap-and-bsde}
		Y_0(x)=y_0(x) ,\quad x \in \bS
	\end{equation*}
	where $y_0$ is according to the backward map~\eqref{eq:bmap}. For almost every $0\le t \le T$:
	\begin{enumerate}
		\item The optimal control $U_t^\opt=0$, $\sP^\nu$-a.s.. 
		\item The optimal state $Y_t \in\clL_t^\nu$ (i.e., $Y_t$ is a
		likelihood ratio). 
		\item The martingale~\eqref{eq:martingale} becomes
		\begin{equation}\label{eq:martingale_with_U0}
		M_t = \clV_t^\nu(Y_t) - \int_0^t \pi_s^\nu(\Gamma Y_s) + \pi_s^\nu\big(|V_s|^2\big) \ud s,\;\;\sP^\nu\text{-a.s.}
		\end{equation}
	\end{enumerate}
\end{proposition}

\begin{proof}
	See Appendix~\ref{apdx:pf-backward-map-and-bsde}.
\end{proof}

\begin{remark}[Variance decay and filter stability]
	\label{rem:cPI}
	The most direct route is to consider a functional inequality as follows:
		\begin{equation}\label{eq:alpha-def}
			\pi_t^\nu(\Gamma Y_t) + \pi_t^\nu(|V_t|^2)
			\ge \alpha_t \clV_t^\nu(Y_t),\; 0\leq t\leq T, \; \sP^\nu\text{-a.s.}
		\end{equation}
		where $\alpha=\{\alpha_t: 0\le t \le T\}$ is a non-negative $\clZ$-adapted process (such a process always exists, e.g., pick  $\alpha = 0$).  The advantage of introducing such a process is the
		following variance decay formula which first appeared in~\cite[Eq.~(8)]{kim2021ergodic} (formula reduces to~\eqref{eq:jensens_ineq} for the choice $\alpha=0$):
		\begin{equation}\label{eq:var_deccay_mg}
			\var^\nu(Y_0(X_0))\leq \E^\nu \left(e^{-\int_0^T\alpha_t \ud t}
			\clV_T^\nu(\gamma_T)\right)
		\end{equation}
	(If~\eqref{eq:alpha-def} holds with equality then so does~\eqref{eq:var_deccay_mg}). Based on the formula~\eqref{eq:var_deccay_mg}, a sufficient condition to show filter stability is to assume $	
	\frac{1}{T} \int_0^T \alpha_t\ud t > c$, $\sP^\nu\text{-a.s.}
	$. Then it is straightforward to show that (see~\cite[Thm.~1]{kim2021ergodic})
	\begin{equation}\label{eq:exp_stability_formula_0}
	\E^\mu\big(\chisq(\pi_T^\mu\mid\pi_T^\nu)\big) \le
	\frac{1}{\underline{a}} e^{-cT} \; \chisq(\mu\mid\nu)
	\end{equation}
	where $\underline{a}:=
	\mathop{\operatorname{essinf}}_{x\in\bS} \gamma_0 (x)$.
	
	While the variance decay formula~\eqref{eq:var_deccay_mg} is attractive, it has been difficult to relate positivity of $\alpha$ to the model properties of the HMM, outside a few special examples described in our prior conference paper~\cite{kim2021ergodic}. A summary of these examples appears in Appendix~\ref{apdx:forward-map-rates} with details in~\cite{kim2021ergodic}.
\end{remark}

\section{Poincar\'e Inequality and filter stability}\label{sec:dual-optimal-ctrl}

\subsection{Poincar\'e inequality (PI) for the filter}

The optimal control system is the
BSDE~\eqref{eq:dual-optimal-control-b} with $U_t$
defined according to optimal the feedback control
law~\eqref{eq:optimal_control_formula}. 

\newP{Optimal control system}
\begin{align}
 -\!\ud Y_t(x) &= \big((\clA Y_t)(x) -  h^\tp (x) \clV_t^\rho(h,Y_t)
                 \nonumber \\ &\qquad + h^\tp (x)(
		V_t(x) -
                \pi_t^\rho(V_t))\big)\ud t - V_t^\tp(x)\ud Z_t \nonumber \\
		 Y_\tau (x)  & = F(x), \;\; x \in \bS, \;\;0\leq t\leq \tau \label{eq:optimal_control_system}
\end{align}

Using the formula~\eqref{eq:optimal_control_formula} for the optimal
control,
\begin{align*}
&\E^\rho(l(Y_t,V_t,U_t^\opt;X_t)|\clZ_t) \nonumber \\
&\qquad = \pi_t^\rho(\Gamma Y_t) + |\clV_t^\rho(h,Y_t)|^2 +
 \clV_t^\rho(V_t),  \quad 0\leq t\leq \tau
\end{align*}
The right-hand side is referred to as the 
  conditional energy.  To define the notion of energy and the
  Poincar\'e constant for the filter, first denote 
\[
\clN:= \big\{\rho\in\clP(\bS): \VV=0 \;\;\forall F\;\in \mathbb{H}_\tau^\rho\big\}
\]

\begin{definition}
Consider~\eqref{eq:optimal_control_system}.  {\em Energy} is defined as follows:
\begin{align*}
\II^\rho(F)&:=\E^\rho \left( \int_0^\tau \pi_t^\rho(\Gamma Y_t) + |\clV_t^\rho(h,Y_t)|^2 +\clV_t^\rho(V_t)  \ud t \right)
\end{align*}
For $\rho\in\clP(\bS)\setminus \clN$, consider 
\[
\beta_\tau^\rho:= \inf\big\{\II^\rho(F) \; : \;F\in \mathbb{H}_\tau^\rho \;\; \& \;\;\VV=1\big\}
\]
and the {\em Poincar\'e constant} is defined as follows:
\[
c^\rho:=\begin{cases}
	\dfrac{1}{\tau}\log\big(1+\beta_\tau^\rho\big),\quad &\rho\in\clP(\bS)\setminus \clN\\
	0,\quad & \rho\in\clN
\end{cases}
\]
\end{definition}

\begin{remark}
The reason for defining the Poincar\'e constant in this manner is that $c^\rho$ then represents a rate.  In particular, using~\eqref{eq:estimator-t-variance}, for each $\rho \in \clP(\bS)\setminus \clN$,
\[
\VV \le e^{-\tau c^\rho} \E^\rho\big(\clV_\tau^\rho(F)\big),\quad \forall\, F\in \mathbb{H}_\tau^\rho 
\]
\end{remark}

\subsection{Analysis of the Poincar\'e constant}

We are interested in existence of the minimizer of
the energy functional $\II^\rho(F)$ for $F\in \mathbb{H}_\tau^\rho$.  If it exists, a minimizer is not unique
because of the following translation symmetry:
\[
\II^\rho(F+ \alpha\ones) = \II^\rho (F)
\]
for any $\clZ_\tau$-measurable random variable $\alpha$ such that
$\tE^\rho (\alpha^2)<\infty$.
For this reason, consider the subspace
\[
{\cal S}^\rho := \{F\in \mathbb{H}_\tau^\rho \;:\;
\pi_\tau^\rho(F)=0, \;\; \sP^\rho-a.s.\}
\]
Then $\cal S^\rho$ is closed subspace. (Suppose $F^{(n)} \to F$ in
$\mathbb{H}_\tau^\rho$ with $\pi^\rho_\tau\big(F^{(n)}\big) = 0$. Then   
$\E^\rho\big(|\pi^\rho_\tau(F)|\big) = \E^\rho\big(|\pi^\rho_\tau(F-F^{(n)})|\big) \leq
\E^\rho\big(\pi^\rho_\tau(|F-F^{(n)}|^2)\big) = \tE^\rho\big(\sigma_\tau^\rho
(|F-F^{(n)}|^2)\big) = \| F-F^{(n)} \|_{\mathbb{H}_\tau^\rho}\to 0$.). 

\begin{proposition}\label{prop:PI_on_S}
Consider the optimal control
        problem~\eqref{eq:dual-optimal-control} with $F\in{\cal S}^\rho$.
        Then
\begin{enumerate}
\item The optimal control $U^\opt=0$. 
\item At time $t=0$, $\rho(Y_0)=0$.
\end{enumerate}
\end{proposition}
\begin{proof}
See Appendix~\ref{apdx:pf-backward-map-and-bsde}.
\end{proof}

Therefore, with $F\in{\cal S}^\rho$, the optimal control system~\eqref{eq:optimal_control_system}  becomes 
\begin{align}
	-\ud Y_t(x) &= \big((\clA Y_t)(x) + h^\tp(x)V_t(x)\big)\ud t -
                      V_t^\tp(x) \ud Z_t, \nonumber\\
	\quad Y_T &= F \in {\cal S}^\rho,\;\; x\in\bS, \;\;0\leq t\leq \tau \label{eq:optimal_control_system_on_S}
\end{align}

Its solution is used to define a linear operator as follows:
\[
\LL_0: {\cal S}^\rho \subset 
\mathbb{H}_\tau^\rho \to  L^2(\rho) \quad \text{by}\quad \LL_0(F) := Y_0
\]
(It is noted that~\eqref{eq:optimal_control_system_on_S} and therefore $\LL_0$ do not depend upon $\rho$ even though
the optimal control system~\eqref{eq:optimal_control_system} does). 
Additional details concerning this operator appear in
Appendix~\ref{apdx:crho_positive} where it is shown that $\LL_0$ is bounded
with $\|\LL_0\|\leq 1$. 

The following Lemma is the main result concerning existence and continuity properties:

\begin{lemma}\label{prop:crho_positive}
Let $\rho \in{\cal P}(\bS) \setminus \clN$. 
Suppose $\LL_0$ is compact.  Then
there exists an $F^\rho \in {\cal S}^\rho $ such that 
\[
\II^\rho(F^\rho) = \beta_\tau^\rho,\;\; \VV=1\;\;\text{where}\;\;Y_0 = \LL_0(F^\rho)
\]
Consider a sequence $\{\rho^{(n)}\in\clP(\bS):n=1,2,\ldots\}$ such that $\rho^{(n)}\ll \rho$. Denote $\gamma^{(n)} := \frac{\ud \rho^{(n)}}{\ud \rho}$ and let $\epsilon_n := \sup_{x\in\bS}|\gamma^{(n)}(x)-1|$. Then
\[
	\lim_{\epsilon_n\to 0}\beta_\tau^{\rho^{(n)}} = \beta_\tau^\rho,\quad \lim_{\epsilon_n\to 0}c^{\rho^{(n)}} = c^\rho
\]  
\end{lemma}

\medskip

\begin{proof}
See Appendix~\ref{apdx:crho_positive}.
\end{proof}

\medskip

\begin{example}[Continued from Ex.~\ref{ex:state_processes}] For the
  examples of the state processes in Ex.~\ref{ex:state_processes}:
\begin{itemize}
\item $\bS=\{1,2,\hdots,d\}$.  $\LL_0$ is compact because $\clY=\Re^d$ is
  finite-dimensional (closed and bounded sets in $\Re^d$ are compact). 
\item $\bS\subseteq \Re^d$. It is conjectured that $\LL_0$ is
  compact whenever $\bS$ is compact subset of $\Re^d$.
\end{itemize}
\end{example}

\medskip

\begin{remark}\label{rem:continuity}
Let $\bS=\{1,2,\hdots,d\}$. It is shown in Appendix~\ref{apdx:pf-finite} that $\clN=\{\delta_s:s\in\bS\}$, the set of $d$ Dirac delta measures ($d$ vertices of the probability simplex $\clP(\bS)\subset \Re^d$).  Combining this with the limit formula in Lemma~\ref{prop:crho_positive} shows that the map $\rho\mapsto c^\rho$ is continuous at all points $\rho$ in the interior of $\clP(\bS)$.  This is because each such $\rho$ admits a neighborhood such that all points in the neighborhood are absolutely continuous with respect to $\rho$.  However, nothing can be said about continuity at the boundary points. 
\end{remark}

\subsection{Main results on variance decay and filter stability}

Fix $\tau>0$. The $\tau$-skeleton of $\{\pi^\nu_t:t\geq 0\}$ is a measure-valued random sequence $\{\pi_{k\tau}^\nu:k=0,1,2,\hdots\}$.  The associated Poincar\'e constants for the skeleton is a real-valued random sequence $\{c^{\pi_{k\tau}^\nu}:k=0,1,2,\hdots\}$.  Define
\[
C_N:=\sum_{k=0}^{N-1} c^{\pi_{k\tau}^\nu},\quad N=1,2,\hdots
\]
The following proposition is the main result that gives the stronger form of the inequality~\eqref{eq:jensens_ineq}.
\begin{proposition}[Variance decay]\label{prop:var_contractive}
	Consider the backward map~\eqref{eq:bmap}.  Then
	\begin{equation}\label{eq:var_contractive}
		\var^\nu(y_0(X_0)) \leq \E^\nu \left( e^{-\tau C_N} \clV_T^\nu(\gamma_T) \right)
		,\quad \forall \;T \ge 0
	\end{equation}
	where $N={\lfloor T/\tau\rfloor}$.
\end{proposition}

\begin{proof}
	See Appendix~\ref{apdx:pf-var_contractive}. 
\end{proof}

Because $\{C_N:N=1,2,\hdots\}$ is non-negative and monotone, define
\[
C_{\infty}(\omega) :=\lim_{N\to\infty} \uparrow C_N(\omega),\quad \omega\in\Omega
\]
where the limit may possibly be $+\infty$.  Based on this definition, the following is the main result on filter stability:

\begin{theorem}[Filter stability]\label{thm:general-nonexponential-case} Suppose $\{\clV_T^\nu(\gamma_T):T\geq 0\}$ is $\sP^\nu$-u.i. and $c^\rho:\clP(\bS)\setminus \clN\to\Re$ is continuous.  Then 
	\begin{enumerate}
		\item[(i)] Either $\sP^\nu([C_{\infty} = \infty]) = 1$, in which case the variance decay property~\eqref{eq:VDP} holds and the filter is stable in $\chisq$-divergence; or
		\item[(ii)] $\sP^\nu([C_{\infty} = \infty]) < 1$, in which case 
		\[
		c^{\pi_T^\nu(\omega)} \stackrel{(T\to\infty)}{\longrightarrow} 0,\quad \sP^\nu\text{-a.e.}\;\omega \in  [C_{\infty}< \infty] 
		\]
	\end{enumerate}
\end{theorem}

\begin{proof}
	See Appendix~\ref{apdx:pf-general}.
\end{proof}

\begin{remark}[Exponential rate]
Let 
\[
  {c}:=\inf \{c^\rho: \rho \in{\cal P}(\bS)\setminus \clN\}
  \] Then it is shown in Appendix~\ref{apdx:pf-var_contractive} (compare with the formula~\eqref{eq:exp_stability_formula_0} in Rem.~\ref{rem:cPI}) that
	\begin{equation}\label{eq:exp_stability_formula}
	\E^\mu\big(\chisq(\pi_T^\mu\mid\pi_T^\nu)\big) \le
	\frac{1}{\underline{a}} e^{-{c}\,(T-\tau)} \;
	\chisq(\mu\mid\nu)
	\end{equation}
        where $\underline{a}=
	\mathop{\operatorname{essinf}}_{x\in\bS} \gamma_0 (x)$.
\end{remark}

\section{Relationship of PI to the model properties}\label{sec:model_prop}

The main task now is to relate the PI to the model properties.  In large part, this program still needs to be carried out.  In this section, some results are described for the finite state-space HMM.

\newP{Assumption 1}  The state-space is finite:
\begin{align*}
	\textbf{(A1)}\qquad	\bS = \{1,2,\hdots,d\}
\end{align*}

\subsection{PI for finite state-space HMM}

We begin with some definitions.  Additional motivation for these can be found in~\cite{kim2021detectable} and~\cite[Ch.~8]{JinPhDthesis}. 
\begin{definition} \label{def:obsvbl} The space of
	{\em observable functions} is the smallest subspace $\clO\subset
	\Re^d$ that satisfies the following two properties:  
	\begin{enumerate}
		\item[(i)] The constant function $\ones\in \clO$; and
		\item[(ii)] If $g\in\clO$ then $\clA g \in \clO$ and $g h
		\in\clO$. 
	\end{enumerate}
	The space of {\em null eigenfunctions} is 
	\[
	S_0 := \{f \in \Re^d \mid \; \Gamma f(x) = 0 \;\;\forall \; x\in\bS \} 
	\]
\end{definition}

These subspaces are useful to define the pertinent model properties for the finite-state HMM as follows:

\begin{definition}\label{def:ergodic}
	\begin{enumerate}
		\item HMM $(\clA,h)$ is {\em observable} if $\clO = \Re^d$. 
		\item The Markov process $\clA$ is {\em ergodic} if
		\begin{equation*}\label{eq:PI_MP_0}
			\Gamma f(x) =0,\;\;\forall \; x\in\bS \implies f(x)= c,\;\;\forall \; x\in\bS
		\end{equation*}
		\item HMM $(\clA,h)$ is {\em detectable} if
		$
		S_0 \subset \clO
		$. 
	\end{enumerate}
\end{definition}

\begin{example}
Consider an HMM on $\bS=\{1,2\}$ with 
\begin{equation*}
\clA = \begin{bmatrix} -\lambda_{12} & \lambda_{12} \\ \lambda_{21} &    -\lambda_{21} \end{bmatrix},\quad
h =  \begin{bmatrix} h(1) \\
  h(2) \end{bmatrix}
\end{equation*}
For this model, the carr\'e du champ operator and the observable space are as follows:
\[
\Gamma f = \begin{bmatrix} \lambda_{12} \\
    \lambda_{21}\end{bmatrix} (f(1)-f(2))^2,\quad \clO = \text{span} \left\{ \begin{bmatrix} 1 \\ 1 \end{bmatrix}, \begin{bmatrix} h(1)\\h(2) \end{bmatrix}\right\}
\]
Consequently,
\begin{enumerate}
\item $\clA$ is 
ergodic iff $(\lambda_{12}+\lambda_{21})>0$.  In this case,
the invariant measure $\bmu = \begin{bmatrix}
  \frac{\lambda_{21}}{(\lambda_{12}+\lambda_{21})} & 
  \frac{\lambda_{12}}{(\lambda_{12}+\lambda_{21})} \end{bmatrix}^\transpose$. 
\item $(\clA,h)$ is  is observable iff $h(1)\neq h(2)$. 
\end{enumerate}
\end{example}

The following proposition gives the relationship between the model properties for finite state-space HMM and the PI.  

\begin{proposition}\label{prop:sufficiency}
	Suppose $\rho\in
	\clP(\bS) \setminus \clN$, and any of
	the following conditions holds:
	\begin{enumerate}
		\item[(i)] $\clA$ is ergodic.
		\item[(ii)] $(\clA,h)$ is observable.
		\item[(iii)] $(\clA,h)$ is detectable.
	\end{enumerate}
	Then $c^\rho>0$.
\end{proposition}

\begin{proof}
	See Appendix~\ref{apdx:pf-main-results-1}.
\end{proof}

\subsection{Filter stability for finite state-space HMM}

\newP{Assumption 2}  The measures $\mu\sim\nu$ (are equivalent) with 
\[
\textbf{(A2)} \qquad 0< \underline{a} := \min_{x\in\bS} \gamma_0(x) \leq  \max_{x\in\bS} \gamma_0(x) =: \bar{a} < \infty
\]

\begin{theorem}\label{thm:finite-case}
  Suppose (A1)-(A2) holds and $(\clA,h)$ is detectable.  Suppose any one of the following conditions hold:
  \begin{enumerate}
  \item[(i)] $\bS=\{1,2\}$.
  \item[(ii)] $c^\rho : \clP(\bS)\setminus\clN \to \Re$ is continuous. 
    \end{enumerate}
  Then the filter is stable in $\chisq$-divergence. 
\end{theorem}

\begin{proof}
	See Appendix~\ref{apdx:pf-finite}.
\end{proof}

\begin{remark}[Contd. from Rem.~\ref{rem:continuity}]
It is shown in \Lemma{prop:crho_positive} that the function $c^\rho$ is continuous at interior points in $\clP(\bS)$. Therefore, the continuity condition ((ii) in \Thm{thm:finite-case}) entails continuity at the boundary points that are not in $\clN$.  For $d=2$, both the boundary points are in $\clN$ and hence the continuity condition is not required. See also Rem.~\ref{rem:continuity_appdx} in Appendix~\ref{apdx:pf-finite}. 
\end{remark}

\section{Discussion and Future Work}\label{sec:conc}

\subsection{Practical significance}

There are two manners in which these results are of practical
significance.  One, our work is important for the analysis  and design of
algorithms for numerical approximation of  the nonlinear filter~\cite{taghvaei2023survey}.
Specifically, the error analysis of these algorithms require estimates
of the two constants related to the exponential decay (the Poincar\'e constant
$c$) and the transient growth
(constant $\frac{1}{\underline{a}}$ in~\eqref{eq:exp_stability_formula})~\cite[Prop. 2]{al2023optimal}.   

The second manner of
practical significance comes from design of reinforcement learning
(RL) algorithms in partially observed settings of the problem.  Many
of these algorithms are based on windowing the past observation data and
using the windowed data as an approximate information state~\cite{subramanian2019approximate,subramanian2022approximate,kao2022common}.  The Poincar\'e
constant is useful to estimate the length of the window for
approximately optimal performance.

\subsection{Future work}

While there are a number of tasks around extending and completing the program begun in~\Sec{sec:model_prop}, it is noted that the definition of backward map~\eqref{eq:bmap} is {\em not} limited to the HMMs with white noise observations (which is the model assumed in all of our work on duality).  This suggests that it may be possible to extend duality and the associated filter stability analysis to a more general class of HMMs.

\section{Acknowledgment}

It is a pleasure to acknowledge Sean Meyn, Amirhossein Taghvaei, and Alain Bensoussan for many useful technical discussions on the topics of duality and stability theory.

\appendix

\section{Appendix}

\subsection{Proof of Proposition~\ref{prop:chisq-stability-implication}} \label{ss:pf-prop71}

Suppose $\mu,\nu\in\clP(\bS)$ and $\mu\ll\nu$. Let $\gamma = \frac{\ud
  \mu}{\ud \nu}$. Then the three forms of $f$-divergence are defined as follows:
	\begin{align*}
		\text{(KL divergence)}\qquad& \kl(\mu\mid \nu) := \int_\bS \gamma(x) \log(\gamma(x))\ud \nu(x)\\
		\text{($\chi^2$ divergence)}\qquad& \chisq(\mu\mid \nu) := \int_\bS (\gamma(x) - 1)^2\ud \nu(x)\\
		\text{(Total variation)}\qquad& \|\mu-\nu\|_\tv := \int_\bS \half |\gamma(x)-1|\ud \nu(x)
	\end{align*}
For these, the following inequalities are standard (see~\cite[Lemma 2.5 and 2.7]{Tsybakov2009estimation}):
\[
2\|\mu-\nu\|_\tv^2 \le \kl(\mu\mid \nu)\le \chisq(\mu\mid\nu)
\]
The first inequality is called the Pinsker's inequality.  The result
follows directly from using these inequalities.  For $L^2$ stability, observe that for any $f \in C_b(\bS)$,
\[
\pi_T^\mu(f) - \pi_T^\nu(f) = \pi_T^\nu(f\gamma_T) - 
\pi_T^\nu(f)\pi_T^\nu(\gamma_T)
\]
Therefore by Cauchy-Schwarz inequality,
\[
|\pi_T^\mu(f)-\pi_T^\nu(f)|^2 \le \frac{\operatorname{osc}(f)}{4} 
\chisq(\pi_T^\mu\mid \pi_T^\nu)
\]
where $\operatorname{osc}(f) = \sup_{x\in\bS} f(x) - \inf_{x\in\bS} f(x)$ denotes the oscillation of 
$f$. Taking $\E^\mu(\cdot)$ on both sides yields the conclusion.

\subsection{Rate bounds for HMM on finite state-space}\label{apdx:forward-map-rates}

A majority of the known bounds for exponential rate of convergence are
for HMMs on finite state-space.  For the ergodic signal model, 
bounds for the stability index $\bar{\gamma}$ (see
Rem.~\ref{rm:contraction-literature}) are tabulated in
Table~\ref{tb:rates} together with references in literature where
these bounds have appeared.  All of these bounds have also been
derived using the approach of this paper.  The bounds are given in terms of
the conditional Poincar\'e constant $c$ (see Rem.~\ref{rem:cPI}) and appear as examples in  our prior
conference paper~\cite{kim2021ergodic}.  

For the non-ergodic signal model, again in finite state-space
settings, additional bounds are known
as follows~\cite[Thm.~7]{atar1997lyapunov}:
	\begin{align*}
		\limsup_{r\to 0}r^2 \overline{\gamma} &\le -\half  \sum_{i \in \bS}\bmu(i)\min_{j\neq i}|h(i)-h(j)|^2\\
		\liminf_{r\to 0}r^2 \overline{\gamma} &\le -\half  \sum_{i,j \in \bS}\bmu(i)|h(i)-h(j)|^2
	\end{align*}
where $r$ is the standard deviation of the measurement noise $W$.
Derivation of these latter pair of bounds using the approach of this paper is open.

\begin{table}
	\centering
	\renewcommand{\arraystretch}{1.2}
	\caption{Rate bounds for finite-state HMM$^\dagger$} 
	\small
	\begin{tabular}{p{0.01\textwidth}p{0.17\textwidth}p{0.12\textwidth}p{0.1\textwidth}}\label{tb:rates}
		\# &	{\bf Bound} 	& {\bf Literature} ($-\bar{\gamma}$) & {\bf Our work}
		($c$) \\
		\hline \hline 
		(1) & $\min_{{i\neq j}} \sqrt{A(i,j)A(j,i)}$ &
		\cite[Thm.~5]{atar1997lyapunov}
		& \cite[Ex.~4]{kim2021ergodic}
		\\
		& & \cite[Thm.~4.3]{baxendale2004asymptotic} &\\
		& & \cite[Corr.~2.3.2]{van2006filtering} & \\[3pt] \hline
		(2) & $\sum_{i\in\bS}
		\bmu(i) \min_{j\neq i}A(i,j)$ &  \cite[Thm.~4.2]{baxendale2004asymptotic} & \cite[Ex.~2]{kim2021ergodic} \\[3pt]
		\hline 
		(3) & $\sum_j \min_{i\neq j} A(i,j)$ &
		~\cite[Ass.~4.3.24]{Moulines2006inference}&
		\cite[Ex.~3]{kim2021ergodic} \\[3pt]
		\hline
		\multicolumn{4}{p{0.45\textwidth}}{$^\dagger$$\{A(i,j):1\leq i,j\leq
          d\}$ is the generator (a transition rate matrix) for the state process and $\bmu$
          is an invariant measure.}
	\end{tabular}
\end{table}

\subsection{Calculation of $\chisq$-divergence}\label{apdx:chisq}

Suppose $\{\pi_t^\mu:t\geq 0\}$ and $\{\pi_t^\nu:t\geq 0\}$ are
the solutions of the nonlinear filtering equation~\eqref{eq:Kushner} starting from prior
$\mu$ and $\nu$, respectively.  Then 
	\begin{equation}\label{eq:forward-equation-chisq}
		\ud \chisq(\pi_t^\mu\mid \pi_t^\nu) = -(\pi_t^\nu\big(\Gamma \gamma_t\big)	+ \clV_t^\mu(\gamma_t,h)\cdot\clV_t^\nu(\gamma_t,h))
		\ud t  + C_t^\tp \ud I_t^\mu
	\end{equation}
where $C_t=\pi_t^\mu\big(\gamma_t (h+\pi_t^\nu(h) - 2\pi_t^\mu(h))\big)$.  With $h=c\ones$, two terms on the right-hand
side are zero 
and the formula~\eqref{eq:rate-chisq-Markov} is obtained.  Before
describing the derivation of~\eqref{eq:forward-equation-chisq}, a
remark concerning the direct use of this equation for the purpose of
filter stability is included as follows:

\begin{remark}\label{rem:chisq_sign_indet}
The term $-\pi_t^\nu(\Gamma \gamma_t)$ on the right-hand
side of~\eqref{eq:rate-chisq-Markov}  is non-positive. However, 
the product
	$\clV_t^\mu(\gamma_t,h) \cdot \clV_t^\nu(\gamma_t,h)$ is sign-indeterminate.  Therefore,
        the equation has not been useful for 
        the asymptotic analysis of the $\chisq$-divergence. 
\end{remark}

\newP{Derivation of~\eqref{eq:forward-equation-chisq}} Using the
equation~\eqref{eq:Kushner} for the filter
\begin{align*}
\ud \chisq(\pi_t^\mu\mid \pi_t^\nu) = \ud \pi_t^\nu (\gamma_t^2) =
  C_{t,1} \ud t + C_{t,2}^\tp \ud I_t^\mu + C_{t,3}^\tp \ud I_t^\nu 
\end{align*}
where the formulae for the three coefficients, obtained through an 
application of the It\^o's formula, are as follows:
\begin{align*}
C_{t,1} & = \pi_t^\nu\big(\Gamma \gamma_t\big)
         +\pi_t^\mu\big(\gamma_t|h-\pi_t^\mu (h)|^2\big) +
        \pi_t^\mu\big(\gamma_t|h-\pi_t^\nu (h)|^2\big)  \\
        & \qquad -2 \: \pi_t^\mu\big(\gamma_t(h-\pi_t^\mu (h))^\tp (h-\pi_t^\nu (h)) \big) \\
C_{t,2} & = 2\: \pi_t^\mu\big(\gamma_t(h-\pi_t^\mu (h))\big),\;\;C_{t,3}= - \pi_t^\nu\big(\gamma_t^2(h-\pi_t^\nu (h))\big)
\end{align*}
Upon noting $\ud I_t^\nu = \ud I_t^\mu +  (\pi_t^\mu(h) -
\pi_t^\nu(h))\ud t$ and simplifying the formula~\eqref{eq:forward-equation-chisq}  for divergence is
obtained.

\subsection{Proof of Theorem~\ref{thm:optimal-solution}}\label{apdx:pf-thm1}

The feedback control formula~\eqref{eq:optimal_control_formula} is
from~\cite[Thm.~3]{duality_jrnl_paper_II}. The equation for the
conditional mean and variance
is proved in~\cite[Prop.~1]{duality_jrnl_paper_II}. 
The SDE~\eqref{eq:variance_Yt_f} for the conditional variance is
derived using the Hamilton's
equation arising from the maximum principle of optimal
control~\cite[Thm.~2]{duality_jrnl_paper_II}.  Specifically, for the
optimal control 
problem~\eqref{eq:dual-optimal-control}, the
co-state process (momentum) is a measure-valued process denoted as
$\{P_t:0\leq t\leq T\}$.  The
Hamilton's equation for momentum is as follows:  For $f\in\clY$,
\begin{align*}
\ud P_t(f) = & \big(P_t(\clA f) + 2 \sigma_t^\rho (\Gamma(f,Y_t))\big)\ud t \\
& + \big(P_t(hf) + 2 U_t^\opt\sigma_t^\rho (f) + 2 \sigma_t^\rho (V_tf)\big)^\tp \ud Z_t
\end{align*}
where $\sigma_t^\rho$ denotes the unnormalized filter at time $t$ (solution of the DMZ
equation starting from initialization $\sigma_0^\rho=\rho$).  
From~\cite[Rem.~5]{duality_jrnl_paper_II}, $\clV_t^\rho (f,Y_t) =
\dfrac{P_t(f)}{2\sigma_t^\rho (\ones)}$.  The SDE~\eqref{eq:variance_Yt_f} is
then obtained by using the It\^o formula. 

An alternate derivation of~\eqref{eq:variance_Yt_f} is based on directly using
the nonlinear filter~\eqref{eq:Kushner} to show that 
\begin{align*}
	\ud \clV_t^\rho&(f,g) = \Big(\pi_t^\rho \big(\Gamma(f,g)\big) + \clV_t^\rho (g,\clA f)\\
	&\qquad+ \clV_t^\rho (f,\clA g) - \clV_t^\rho (h,f)\clV_t^\rho (h,g)\Big)\ud t\\
	&+ \Big(\clV_t^\rho (h,fg)-\pi_t^\rho (f)\clV_t^\rho (h,g)-\pi_t^\rho (g)\clV_t^\rho (h,f)\Big)^\tp \ud I_t^\rho
\end{align*}
With $g = Y_t$, using the BSDE~\eqref{eq:dual-optimal-control-b} with
$U_t = U_t^\opt$, upon simplifying, again
yields~\eqref{eq:variance_Yt_f}.

\subsection{Proof of~\Prop{prop:backward-map-and-bsde} and~\Prop{prop:PI_on_S}}\label{apdx:pf-backward-map-and-bsde}

Suppose $\pi_\tau^\rho(F) = c$ where $c$ is a deterministic constant.  Using~\eqref{eq:estimator-t}, because $\tsP^\rho\sim \sP^\rho$,
\[
c = \rho(Y_0) - \int_0^\tau \big(U_t^\opt\big)^\tp \ud Z_t,\quad \tsP^\rho\text{-a.s.}
\]
By the uniqueness of the It\^o
representation, then 
\[
U_t^\opt = 0, \quad \text{a.e.} \; 0\leq t\leq \tau ,\;\; \tsP^\rho\text{-a.s.}
\]
and, because these are equivalent, also $\sP^\rho\text{-a.s.}$. 
Using~\eqref{eq:estimator-t}, this also gives
\[
\pi_t^\rho (Y_t) = c, \quad \sP^\rho\text{-a.s.}\;\; \text{a.e.} \; 0\leq t\leq \tau 
\]
and
$
\E^\rho\big(\clV_t^\rho (Y_t)\big) = \E^\rho ( |Y_t(X_t)-c|^2)  = \var^\rho\big(Y_t(X_t)\big)
$.

Now consider the stochastic process $\{Y_t(X_t):0\leq
t\leq \tau \}$.  Because $U=0$, the It\^o-Wentzell formula is used to show
that (see~\cite[Appdx.~A]{duality_jrnl_paper_I})
\[
\ud Y_t(X_t) = V_t^\tp(X_t) \ud W_t + \ud N_t,\quad 0\leq t\leq
\tau 
\]    
where $\{N_t:t\geq 0\}$ is a $\sP^\rho$-martingale. Integrating this from $t$ to $\tau $ yields 
\[
F(X_\tau)  = Y_t(X_t) + \int_t^\tau  V_s^\tp(X_s) \ud W_s + \ud
N_s
\]  
which gives
\[
Y_t(x) = \E^\rho\big(F(X_T)\mid\clZ_t \vee [X_t=x]\big),\;\; x\in\bS, \;\; \sP^\rho\text{-a.s.}
\]

\begin{proof}[of~\Prop{prop:backward-map-and-bsde}]
	Set $\rho=\nu$ and $\tau=T$.  
	If $F \in \clL_T^\nu$ the representation as a conditional expectation shows $Y_t(x) \geq 0$, and because $\pi_t^\nu(Y_t) = 1$, $Y_t$ is a
	likelihood ratio.  For $F=\gamma_T$,
	\[
	Y_0(x) = \E^\nu(\gamma_T(X_T)| [X_0=x]),\quad x\in\bS
	\]
	The right-hand side is the backward map~\eqref{eq:bmap} which
	proves $Y_0(x)=y_0(x)$.  
\end{proof}

\subsection{Proof of Lemma~\ref{prop:crho_positive}}\label{apdx:crho_positive}

For $F\in  \mathbb{H}_\tau^\rho$, we begin by noting
\begin{equation}\label{label:norm_equiv}
\|F\|_{\mathbb{H}_\tau^\rho}^2= \tE^\rho \big( \sigma_\tau^\rho(F^2)\big) =
\E^\rho \big( \pi_\tau^\rho(F^2)\big) = \E^\rho \big( (F(X_\tau))^2\big)
\end{equation}

Consider the optimal control
system~\eqref{eq:optimal_control_system_on_S}. 
Define its solution operator 
\[
\LL:
{\cal S}^\rho\subset \mathbb{H}_\tau^\rho \to \mathbb{H}^\rho ([0,\tau])\times \mathbb{H}^\rho ([0,\tau])^m\;\;
\text{by}\;\; \LL(F) :=
  (Y,V) 
\]
As with $\LL_0$, this operator too does not depend upon $\rho$.  

Because $U^\opt=0$, with $(Y,V) = \LL(F)$, the formula for energy
becomes 
\[
\II^\rho(F) =  \E^\rho \left( \int_0^\tau \pi_t^\rho(\Gamma
  Y_t) + \pi_t^\rho(|V_t|^2) 
             \ud t \right),\quad F\in{\cal S}^\rho
\]
where note $\pi_t^\rho(|V_t|^2) := \int_\bS V_t^\tp(x) V_t(x) \ud \pi_t^\rho(x) $.  
The optimality
equation~\eqref{eq:estimator-t-variance} gives
\begin{equation}\label{eq:opt_eqn_crhopos}
\rho(Y_0^2) + \II^\rho (F) = \E^\rho ((F(X_\tau))^2),\quad F\in{\cal S}^\rho
\end{equation}
This shows that $\LL_0:{\cal S}^\rho \to L^2(\rho)$ is a bounded
operator with $\|\LL_0\|\leq 1$. 
Because $\rho(Y_0)=0$, the definition of $\beta_\tau^\rho$ becomes
\[
\beta_\tau^\rho=\inf\{\II^\rho(F) \;:\; F\in {\cal S}^\rho, \; Y_0 = \LL_0(F)\;\; \& \;\; \rho(Y_0^2)=1\}
\]
To obtain the minimizer, setting $(\tilde Y,\tilde V) = \LL(\tilde F)$ for $\tilde F\in {\cal
  S}^\rho $, the functional derivative is evaluated as
follows:
\[
\langle \nabla \II^\rho(F) , \tilde F \rangle := 2
\E^\rho \left( \int_0^\tau \pi_t^\rho(\Gamma(Y_t,\tilde Y_t) +
  \pi_t^\rho( V_t^\tp \tilde V_t) 
             \ud t \right)
\]
where note $\pi_t^\rho(V_t^\tp \tilde{V}_t) := \int_\bS V_t^\tp(x) \tilde{V}_t(x) \ud \pi_t^\rho(x) $.  From the Cauchy-Schwarz formula,
using~\eqref{label:norm_equiv} and \eqref{eq:opt_eqn_crhopos},
\begin{equation*}\label{eq:bdd_lin_fn}
|\langle \nabla \II^\rho(F) , \tilde F \rangle|^2 \leq 4
\II^\rho (F) 
\; \|\tilde F\|_{\mathbb{H}_\tau^\rho}^2
\end{equation*}
This shows that $\tilde F \mapsto \langle \nabla
\II^\rho(F) , \tilde F\rangle$ is a bounded linear functional as a map
from 
${\cal S}^\rho \subset \mathbb{H}_\tau^\rho$ into $\Re$. 
With these formalities completed, we show a minimizer exists.

\begin{proof}[of Lemma~\ref{prop:crho_positive} (existence)] 
See Appendix~\ref{appdx:lemma_crho_exist}.  
\end{proof}

For a fixed $\rho \in \clP(\bS)$, the minimizer is denoted by
$F^\rho$ with $\II^\rho(F^\rho) = \beta_\tau^\rho$.  From the proof of existence in
Appendix~\ref{appdx:lemma_crho_exist}, we also have
$\|F^\rho\|_{\mathbb{H}_\tau^\rho}^2 = 1  + \beta_\tau^\rho$. We now show continuity.

\begin{proof}[of Lemma~\ref{prop:crho_positive} (continuity)]
See Appendix~\ref{appdx:lemma_crho_cont}.  
\end{proof}

\subsection{Proof of Lemma~\ref{prop:crho_positive} (existence)}\label{appdx:lemma_crho_exist}

Consider an infimizing sequence $\{F^{(n)} \in
\clS^\rho  :n=1,2,\hdots\}$ such that
$\II^\rho\big(F^{(n)}\big)\to \beta_\tau^\rho$ and $\rho\big((Y_0^{(n)})^2\big)  =1$ with $Y_0^{(n)}:=\LL_0\big(F^{(n)}\big)$.  The proof is obtained in three steps:

\newP{Step 1} Establish a limit $F^\rho\in \clS^\rho $ such that $F^{(n)}$ converges
  weakly (in $\mathbb{H}_\tau^\rho$) to $F^\rho$.  The weak
  convergence is denoted as $
F^{(n)} \rightharpoonup F^\rho$.

\newP{Step 2} Show that $\II^\rho(F^\rho) = \beta_\tau^\rho$.

\newP{Step 3}. Set $Y_0=\LL_0(F^\rho)$. Show that $\rho(Y_0)=0$, $\rho(Y_0^2)=1$.

We begin with step 1. 
To establish a limit, use the optimality
equation~\eqref{eq:opt_eqn_crhopos}:
\[
\rho\big((Y_0^{(n)})^2\big) + \II(F^{(n)}) = \E^\rho \big((F^{(n)}
(X_\tau))^2\big),\quad n=1,2,\hdots
\]
Now $\rho\big((Y_0^{(n)})^2\big)  =1$ and because $\II(F^{(n)}) \to \beta_\tau^\rho$,
by considering a sub-sequence if necessary, using~\eqref{label:norm_equiv},
\[
\|F^{(n)}\|_{\mathbb{H}_\tau^\rho}^2 = \E^\rho ((F^{(n)}
(X_\tau))^2) < 1+ (\beta_\tau^\rho+1) 
\]
We thus have a bounded sequence in the Hilbert space $\mathbb{H}_\tau^\rho$.
Therefore, there exists a weak limit $F^\rho\in \mathbb{H}_\tau^\rho$ such
that $
F^{(n)} \rightharpoonup F^\rho$.  Because $\clS^\rho$ is closed, $F^\rho\in
\clS^\rho$.  
This completes the proof of step 1.

Next we show $\II^\rho(F^\rho) = \beta_\tau^\rho$.  Because the map from $F^\rho\mapsto \II^\rho(F^\rho)$ is
convex, we have
\[
\II^\rho(F^{(n)}) \geq \II^\rho(F^\rho) +  \langle \nabla \II^\rho(F^\rho) ,
(F^{(n)} - F^\rho) \rangle
\]
We have already shown that $\tilde F \mapsto \langle \nabla
\II^\rho(F^\rho) , \tilde F\rangle$ is a bounded linear functional.
Therefore, letting $n\to \infty$, the second term on the right-hand
side converges to zero and 
\[
\lim_{n\to\infty} \II^\rho(F^{(n)}) \geq \II^\rho(F^\rho)
\]
This property of the functional is referred to as weak lower semi-continuity.  Because
$\II^\rho(F^{(n)}) \to \beta_\tau^\rho$, we have $\II^\rho(F^\rho) \leq \beta_\tau^\rho$.
However, $\beta_\tau^\rho$ is the infimum.  It therefore must be that $\II^\rho(F^\rho) =
\beta_\tau^\rho$.  This completes the proof of the step 2.

The step 3 of the proof is to show that setting  
$Y_0:=\LL_0(F^\rho)$ gives $\rho(Y_0)=0$ and $\rho(Y_0^2)=1$.  
This is where the assumption on compactness of $\LL_0$ is used.  Because $
F^{(n)} \rightharpoonup F^\rho$ in $\mathbb{H}_\tau^\rho$ and $\LL_0$ is compact, we have $Y_0^{(n)}
\to Y_0$ in $L^2(\rho)$.  Then $|\rho(Y_0)|  = |\rho(Y_0-Y_0^{(n)})| \leq
\rho\big(|Y_0-Y_0^{(n)}|^2\big) \to 0$ and $\rho(Y_0^2) = \lim_{n\to\infty}
\rho\big((Y_0^{(n)})^2\big) =1$ by the continuity of the norm with respect to strong
convergence.

From~\eqref{eq:opt_eqn_crhopos}, it also follows that 
\[
\|F^\rho\|_{\mathbb{H}_\tau^\rho}^2 = \rho(Y_0^2) + \II^\rho (F^\rho)  = 1 + \beta_\tau^\rho
\]
and thus $\|F^{(n)}\|_{\mathbb{H}_\tau^\rho} \to
\|F^\rho\|_{\mathbb{H}_\tau^\rho}$.  Therefore, $F^\rho$ is in fact a
strong limit whereby $F^{(n)} \to F^\rho$ strongly in
$\mathbb{H}_\tau^\rho$. 

\subsection{Proof of Lemma~\ref{prop:crho_positive} (continuity)}\label{appdx:lemma_crho_cont}

Let $\epsilon_n\to 0$ as $n\to\infty$.  Our goal is to show
\[
\beta_\tau^\rho \leq \liminf_{n\to \infty} \beta_\tau^{\rho^{(n)}} \leq
\limsup_{n\to \infty} \beta_\tau^{\rho^{(n)}} \leq \beta_\tau^\rho
\]
W.l.o.g., we assume $\epsilon_{n}<\frac{1}{2}$, $\forall \; n$.  The following technical result is helpful for the proof.  

\begin{proposition}\label{prop:continuity_of_norm_and_energy}
	The following holds:
	\[
	(1 - \epsilon_n) 
	\| F\|^2_{\mathbb{H}_\tau^{\rho}}  \leq \| F\|^2_{\mathbb{H}_\tau^{\rho^{(n)}}} 
	\leq (1 + \epsilon_n) 
	\| F\|^2_{\mathbb{H}_\tau^{\rho}} 
	\]
	Consequently, $F \in \mathbb{H}_\tau^{\rho}$
	iff $F \in
	\mathbb{H}_\tau^{\rho^{(n)}}$. For $F \in \mathbb{H}_\tau^{\rho}$,
	\[
	\lim_{n\to \infty} \II^{\rho^{(n)}} (F) \to \II^{\rho} (F)
	\]
\end{proposition}

\begin{proof}[of~\Prop{prop:continuity_of_norm_and_energy}]
	We have
	\[
	\|F\|^2_{\mathbb{H}_\tau^{\rho^{(n)}}} = \E^{\rho^{(n)}} (|F(X_\tau)|^2 ) = \E^{\rho} \big(\gamma^{(n)}(X_0) |F(X_\tau)|^2\big)
	\]
	Because $(1-\epsilon_n)\leq \gamma^{(n)}(X_0) \leq (1 +\epsilon_n)$, $\sP^{\rho}$-a.s., the equivalence of norm follows.
	Next, the continuity of the functional is shown. Let $F\in\mathbb{H}_\tau^{\rho}$.  By translation symmetry,
	\begin{align*}
		\II^{\rho^{(n)}}(F) &= \II^{\rho^{(n)}}(F - \pi_\tau^{\rho^{(n)}}(F)),\\ 
		\II^{\rho}(F) &= \II^{\rho}(F - \pi_\tau^{\rho}(F))
	\end{align*}
	Let $\tilde{F}^{(n)} := (\pi_\tau^{\rho}(F)-\pi_\tau^{\rho^{(n)}}(F))\ones$ and denote
	\begin{align*}
		(Y^{(n)},V^{(n)})&:=\LL(F- \pi_\tau^{\rho^{(n)}}(F))\\
		(Y,V) &:=\LL(F - \pi^{\rho}(F))\\
		(\tilde{Y}^{(n)},\tilde{V}^{(n)})&:= \LL (\tilde{F}^{(n)})
	\end{align*}
	Then $(Y^{(n)},V^{(n)}) = (Y,V) + (\tilde{Y}^{(n)},\tilde{V}^{(n)})$.  Denote
        	\begin{align*}
		S^{(n)}&:= \int_0^\tau \Gamma Y_t^{(n)}(X_t) + |V_t^{(n)}(X_t)|^2 \ud t \\
		S&:= \int_0^\tau \Gamma Y_t(X_t) + |V_t(X_t)|^2 \ud t \\
		\tilde{S}^{(n)}&:= \int_0^\tau \Gamma \tilde{Y}_t^{(n)}(X_t) + |\tilde{V}_t^{(n)}(X_t)|^2 \ud t 
	        \end{align*}
                Then 
	\begin{align}
		& \E^\rho \big(S^{(n)}\big) = \E^\rho (S) +  \E^\rho \big(\tilde{S}^{(n)}\big) \label{eq:sn_cross_term}\\
		& + 2 \E^\rho \Big( \int_0^\tau \Gamma (Y_t, \tilde{Y}_t^{(n)})(X_t) + (V_t(X_t))^\tp( \tilde V_t^{(n)}(X_t)) \ud t  \Big) \nonumber
	\end{align}
	Meanwhile,
	\begin{align*}
		\II^{\rho}(F) & - \II^{\rho^{(n)}}(F)  = \E^\rho (S) - \E^{\rho^{(n)}}\big(S^{(n)}\big)\\
		& =
		\underbrace{\big(\E^\rho (S) - \E^\rho (S^{(n)})\big)}_{\text{term (i)}} + \underbrace{\big(\E^\rho (S^{(n)})-\E^{\rho^{(n)}}(S^{(n)})\big)}_{\text{term (ii)}}
	\end{align*}
	It is shown that each of the two terms are $O(\epsilon_n)$.  For term (ii),
	\begin{align*}
		&|\E^\rho (S^{(n)})-\E^{\rho^{(n)}}(S^{(n)})| \leq  \E^{\rho^{(n)}} \left(\left| \frac{1}{\gamma^{(n)}(x)} -1 \right| S^{(n)}\right) \\ &\leq \frac{\epsilon_n}{1-\epsilon_n} \II^{\rho^{(n)}}(F) 
		\leq \frac{\epsilon_n}{1-\epsilon_n} \| F\|^2_{\mathbb{H}_\tau^{\rho^{(n)}}} \leq \epsilon_n \frac{1+\epsilon_n}{1-\epsilon_n} \| F\|^2_{\mathbb{H}_\tau^{\rho}}
	\end{align*}
	For term~(i), using Cauchy-Schwarz in~\eqref{eq:sn_cross_term},
	\begin{align}
		|\E^\rho (S) - \E^\rho (S^{(n)})| & \leq \E^\rho (\tilde{S}^{(n)}) + 2 \sqrt{\II^\rho(F)} \sqrt{\E^\rho(\tilde{S}^{(n)})}\nonumber\\
		& \leq \E^\rho (\tilde{S}^{(n)}) + 2 \| F\|_{\mathbb{H}_\tau^{\rho}} \sqrt{\E^\rho(\tilde{S}^{(n)})}\label{eq:sn_cross_term_1}
	\end{align}
	Now, using the Bayes' formula, 
	\begin{align*}
		\tilde{F}^{(n)} & = (\pi^{\rho}(F)-\pi^{\rho^{(n)}}(F))\ones \\
		& = \left(
		\E^{\rho} (F(X_\tau)|\clZ_\tau) - \frac{\E^{\rho} (\gamma^{(n)}(X_0) F(X_\tau)|\clZ_\tau)}{\E^{\rho} (\gamma^{(n)}(X_0)|\clZ_\tau)} 
		\right)
		\ones 
	\end{align*}
	Now, because $1-\epsilon_n \leq \gamma^{(n)}(x) \leq 1+\epsilon_n$, 
	\[
	\left|1 - \frac{\gamma^{(n)}(X_0)}{\E^{\rho} (\gamma^{(n)}(X_0)|\clZ_\tau)} \right| \leq \frac{2\epsilon_n}{1 - \epsilon_n},\quad \sP^{\rho}-\text{a.s.}
	\]
	and thus
	\begin{align*}
		\| \tilde F^{(n)}\|^2_{\mathbb{H}_\tau^{\rho}} =  \E^{\rho} (|\tilde{F}^{(n)} (X_\tau)|^2) \leq 4 \frac{\epsilon_n^2}{(1-\epsilon_n)^2} \| F\|^2_{\mathbb{H}_\tau^{\rho}} 
	\end{align*}
	Finally, because $(\tilde{Y}^{(n)},\tilde{V}^{(n)}):= \LL (\tilde{F}^{(n)})$, 
	\[
	{\rho}((\tilde{Y}_0^{(n)})^2) + \E^\rho (\tilde{S}^{(n)}) =  \| \tilde F^{(n)}\|^2_{\mathbb{H}_\tau^{\rho}} \leq 4 \frac{\epsilon_n^2}{(1-\epsilon_n)^2} \| F\|^2_{\mathbb{H}_\tau^{\rho}} 
	\]
	Substituting the estimate in~\eqref{eq:sn_cross_term_1},
	\[
	|\E^\rho (S) - \E^\rho (S^{(n)})| \leq 4 \frac{\epsilon_n}{(1-\epsilon_n)^2} \| F\|^2_{\mathbb{H}_\tau^{\rho}}
	\]
	which shows that term (ii) is also $O(\epsilon_n)$.  Combining the estimates for the two terms,
	\[
	|\II^{\rho}(F)  - \II^{\rho^{(n)}}(F)| \leq \epsilon_n \frac{5-\epsilon_n^2}{(1-\epsilon_n)^2} \| F\|^2_{\mathbb{H}_\tau^{\rho}}
	\]
	which proves the continuity of the functional. 
\end{proof}

The continuity of the map $\rho\mapsto \beta_\tau^\rho$ is shown in two steps:

\newP{Step 1. Proof of $\limsup_{n\to \infty} \beta_\tau^{\rho^{(n)}} \leq
	\beta_\tau^\rho$} For $\rho$, consider a minimizer
$F^{\rho} \in \mathbb{H}_\tau^{\rho}$ such that
$\beta_\tau^\rho = \II^{\rho} (F^{\rho} )$. 
From \Prop{prop:continuity_of_norm_and_energy}, $F^{\rho}
\in\mathbb{H}_\tau^{\rho^{(n)}}$ and because $\beta_\tau^{\rho^{(n)}}$ is the
minimum value,
\[
\beta_\tau^{\rho^{(n)}} \leq \II^{\rho^{(n)}}(F^{\rho})
\]
Letting $n\to \infty$, from~\Prop{prop:continuity_of_norm_and_energy},
the right-hand side converges to
$\II^{\rho}(F^{\rho})$ which gives
\[
\limsup_{n\to \infty} \beta_\tau^{\rho^{(n)}} \leq
\II^{\rho}(F^{\rho}) =  \beta_\tau^\rho
\]

\newP{Step 2. Proof of $\beta_\tau^\rho \leq \liminf_{n\to \infty}
	\beta_\tau^{\rho^{(n)}}$} For $\rho^{(n)}$, consider a minimizer
$F^{\rho^{(n)}} \in \clS^{\rho^{(n)}} \subset \mathbb{H}_\tau^{\rho^{(n)}}$ such that 
$\beta_\tau^{\rho^{(n)}} = \II^{\rho^{(n)}} (F^{\rho^{(n)}} )$ and with $Y_0^{\rho^{(n)}} := \LL_0(F^{\rho^{(n)}} - \pi_\tau^{\rho^{(n)}} (F^{\rho^{(n)}}))$, $\rho^{(n)}((Y_0^{\rho^{(n)}})^2)=1$. From the estimate in step~1, upon
considering a subsequence if necessary,  
\[
\| F^{\rho^{(n)}}\|_{\mathbb{H}_\tau^{\rho^{(n)}}}^2 = 1  + 
\beta_\tau^{\rho^{(n)}} \leq  2  + \beta_\tau^\rho \quad \forall \; n
\]
From~\Prop{prop:continuity_of_norm_and_energy}, because $\epsilon_n \leq \frac{1}{2}$, the subsequence is bounded also in $\mathbb{H}_\tau^{\rho}$. 
Set
\[
F^{(n)} := F^{\rho^{(n)}} - \pi_\tau^{\rho} (F^{\rho^{(n)}}),\;\;   Y_0^{(n)}  = \LL_0(F^{(n)})
\]
Then $F^{(n)} \in \clS^{\rho}$.  Conclude a weak limit $F\in \clS^{\rho} \subset \mathbb{H}_\tau^{\rho}$ such that $F^{(n)} \rightharpoonup F$ (in $\mathbb{H}_\tau^{\rho}$). Because $\LL_0:\mathbb{H}_\tau^{\rho}\to L^2(\rho)$ is compact, denoting $Y_0:=\LL_0(F)$,
\[
\rho(Y_0^2) = \lim_{n\to\infty} \rho((Y_0^{(n)})^2)
\]
We make two claims as follows:
\begin{align*}
	\text{(Claim A)}&\quad
	\| F^{(n)} \|^2_{\mathbb{H}_\tau^{\rho}}
	= \| F^{\rho^{(n)}}\|_{\mathbb{H}_\tau^{\rho^{(n)}}}^2 + O(\epsilon_n)\\
	\text{(Claim B)}&\quad \rho((Y_0^{(n)})^2) = {\rho^{(n)}}((Y_0^{\rho^{(n)}})^2) + O(\epsilon_n) 
\end{align*}
From Claim B, $\rho(Y_0^2) = 1$.  Now, it is a property of weak convergence that
\[
\| F \|^2_{\mathbb{H}_\tau^{\rho}} \leq \liminf_{n\to\infty} \| F^{(n)} \|^2_{\mathbb{H}_\tau^{\rho}}
\]
From Claim A,
\[
\| F \|^2_{\mathbb{H}_\tau^{\rho}} \leq \liminf_{n\to\infty} \|F^{\rho^{(n)}}\|_{\mathbb{H}_\tau^{\rho^{(n)}}}^2 + O(\epsilon_n) 
\]
We have
\begin{align*}
	\| F \|^2_{\mathbb{H}_\tau^{\rho}} & = \rho(Y_0^2)  + \II^{\rho}(F) = 1 + \II^{\rho}(F) \\
	\| F^{\rho^{(n)}}\|_{\mathbb{H}_\tau^{\rho^{(n)}}}^2 & = {\rho^{(n)}}((Y_0^{\rho^{(n)}})^2) + 
	\II^{\rho^{(n)}}(F^{\rho^{(n)}})  = 1 + 
	\beta_\tau^{\rho^{(n)}}
\end{align*}
Combining
\[
1 + c^{\rho} \leq 1 + \II^{\rho}(F) \leq \liminf_{n\to\infty} (1  + 
\beta_\tau^{\rho^{(n)}} + O(\epsilon_n))
\]
which shows $\beta_\tau^\rho \leq \liminf_{n\to \infty}
\beta_\tau^{\rho^{(n)}}$.  It remains to prove the two claims.

\newP{Proof of Claim A} Let
$
\tilde{F}^{(n)} := (\pi_\tau^{\rho^{(n)}} (F^{\rho^{(n)}}) - \pi_\tau^{\rho} (F^{\rho^{(n)}}))\ones
$. 
Then, by repeating the argument in the proof of~\Prop{prop:continuity_of_norm_and_energy},
\[
\| \tilde{F}^{(n)}\|_{\mathbb{H}_\tau^{\rho}}^2 \leq 4 \frac{\epsilon_n^2}{(1-\epsilon_n)^2} \|F^{\rho^{(n)}}\|_{\mathbb{H}_\tau^{\rho}}^2 = O(\epsilon_n^2)
\]
Because $F^{(n)} = F^{\rho^{(n)}} + \tilde{F}^{(n)}$, the claim is proved from using~\Prop{prop:continuity_of_norm_and_energy}. 

\newP{Proof of Claim B} Let $\tilde{Y}_0^{(n)}:= \LL_0 (\tilde{F}^{(n)})$. Then $
{\rho}((\tilde{Y}_0^{(n)})^2) \leq  \| \tilde{F}^{(n)}\|_{\mathbb{H}_\tau^{\rho}}^2$. 
Because $Y_0^{(n)}  = Y_0^{\rho^{(n)}} + \tilde{Y}_0^{(n)}$,
\begin{align*}
	\rho((Y_0^{(n)})^2) & = \rho((Y_0^{\rho^{(n)}})^2) + O(\epsilon_n) = {\rho^{(n)}}((Y_0^{\rho^{(n)}})^2) + O(\epsilon_n)
\end{align*}
which concludes the proof of the claim.

\subsection{Proof of Proposition~\ref{prop:var_contractive}}\label{apdx:pf-var_contractive}

The proof is based on the following technical Lemma:

\begin{lemma}\label{lem:translation}
	Let $(Y,V)$ be the solution of the optimal control
	system~\eqref{eq:optimal_control_system} with $\rho=\nu$, $\tau=T$
	and $Y_T=\gamma_T$.  Then 
	\begin{align*}
		\E^\nu\Big(  \int_t^{t+\tau}
		\pi_s^\nu (\Gamma Y_s)  +	\pi_s^\nu (|V_s|^2)
		\ud s& \mid \clZ_t\Big)
		\ge \beta_\tau^{\pi_t^\nu} \clV_{t}^\nu(Y_t),\\
                & 0\leq t \leq T-\tau,\;\;\sP^\nu\text{-a.s.}
	\end{align*}
\end{lemma}

\begin{proof}
	See Appendix~\ref{apdx:fbsde}.
\end{proof}

From Lemma~\ref{lem:translation}, using the formula~\eqref{eq:martingale_with_U0} for martingale in \Prop{prop:backward-map-and-bsde}, for $0\leq t\leq T-\tau$,
	\begin{align*}
		\beta_\tau^{\pi_t^\nu} \clV_{t}^\nu(Y_t) & \le \E^\nu\Big( \int_{t}^{t+\tau}
		\pi_s^\nu (\Gamma Y_s) + \pi_s^\nu(|V_s|^2)
		\ud s \mid \clZ_{t}\Big) \\
                & = \E^\nu\big(\clV_{t+\tau}^\nu(Y_{t+\tau})\mid\clZ_{t}\big) - \clV_{t}^\nu(Y_{t}),\;\;\sP^\nu\text{-a.s.}
	\end{align*}
Therefore, using the definition of the Poincar\'e constant,
	\begin{equation}\label{eq:inc_PI_identity}
	\E^\nu\big(\clV_{t+\tau}^\nu(Y_{t+\tau})\mid\clZ_t\big)\ge e^{\tau c^{\pi_t^\nu}}\clV_t^\nu(Y_t),\;\;\sP^\nu\text{-a.s.}
	\end{equation}

In both the proof of~\Prop{prop:var_contractive} and in derivation of~\eqref{eq:exp_stability_formula}, let $N = \lfloor T/\tau\rfloor$ and partition the interval $[0,T]$ as $0=t_0<\hdots < t_k < \hdots < t_{N+1} = T$ where $t_k = k \tau$ for $k=0,1,2,\hdots,N$.  

\begin{proof}[of \Prop{prop:var_contractive}] For the partition, formula~\eqref{eq:inc_PI_identity} gives
\[
\E^\nu\big(\clV_{t_{k+1}}^\nu(Y_{t_{k+1}})\mid\clZ_{t_k}\big)\ge e^{\tau c^{\pi_{t_k}^\nu}}\clV_{t_k}^\nu(Y_{t_k}),\;\;\sP^\nu\text{-a.s.},\;k<N
\]
and therefore,
\begin{align*}
	\E^\nu\big(e^{-\tau C_{k}}&\clV_{t_k}^\nu(Y_{t_k})\big)\\
	 &\le \E^\nu\big(e^{-\tau C_{k}}e^{-\tau c^{\pi_{t_k}^\nu}}\E^\nu\big(\clV_{t_{k+1}}^\nu(Y_{t_{k+1}})\mid\clZ_{t_{k}}\big)\big)\\
	 &=\E^\nu\big(e^{-\tau C_{k+1}}\clV_{t_{k+1}}^\nu(Y_{t_{k+1}})\big),\quad k<N
\end{align*}
A recursive application of this identity gives
	\begin{align*}
	&	\var^\nu(Y_0(X_0)) = \clV_{t_0}^\nu (Y_{t_0}) \le \E^\nu\big(e^{-\tau C_1}\clV_{t_1}^\nu(Y_{t_1})\big)\\
		&\le \E^\nu\big(e^{-\tau C_2}\clV_{t_2}^\nu(Y_{t_2})\big)\le \cdots \le \E^\nu\big(e^{-\tau C_N}\clV_{t_N}^\nu(Y_{t_N})\big)
	\end{align*}
        Therefore,
        \begin{align*}
          & \var^\nu(Y_0(X_0))  \leq \E^\nu\big(e^{-\tau C_N}\clV_{t_N}^\nu(Y_{t_N})\big)\\
          &\leq \E^\nu\big(e^{-\tau C_N}\E^\nu(\clV_T^\nu(Y_T)\mid\clZ_{t_N})\big) = \E^\nu\big(e^{-\tau C_N}\clV_T^\nu(Y_T)\big)
        \end{align*}
        which concludes the result because $Y_T=\gamma_T$ and $Y_0=y_0$ (from \Prop{prop:backward-map-and-bsde}). 
\end{proof}

\newP{Proof of formula~\eqref{eq:exp_stability_formula}} From the definition of ${\cal N}$, we have
\begin{align*}
  c^{\pi_t^\nu} & \geq c, \quad \pi_t^\nu \in \clP(\bS)\setminus {\cal N}\\
  \clV_t^\nu(Y_t) & = 0, \quad \pi_t^\nu \in  {\cal N}
\end{align*}
Therefore, for the partition, formula~\eqref{eq:inc_PI_identity} gives
\[
\E^\nu\big(\clV_{t_{k+1}}^\nu(Y_{t_{k+1}})\mid\clZ_{t_k}\big)\ge e^{\tau c}\clV_{t_k}^\nu(Y_{t_k}),\;\;\sP^\nu\text{-a.s.},\;k<N
\]
and upon taking expectations of both sides
\[
\var^\nu(Y_{t_{k+1}}(X_{t_{k+1}})) \geq  e^{\tau c} \var^\nu(Y_{t_{k}}(X_{t_{k}})),\quad k<N
\]
A recursive application of this identity gives
\begin{align*}
\var^\nu(Y_0(X_0)) & \leq e^{-c \tau N} \var^\nu(Y_{t_N}(X_{t_N}))\\ & \leq e^{-c (T-\tau)}  \var^\nu(\gamma_T(X_T))
\end{align*}
Meanwhile, from~\eqref{eq:chisq_identity_intro},
\begin{align*}
 \big(\E^\mu\big(\chisq(\pi_T^\mu\mid\pi_T^\nu)\big)\big)^2 & \le \var^\nu(\gamma_0(X_0))\var^\nu(Y_0(X_0)) \\
\le & e^{-c (T-\tau)} \var^\nu(\gamma_0(X_0))\var^\nu\big(\gamma_T(X_T)\big) 
\end{align*}
Since $\var^\nu\big(\gamma_T(X_T)\big) =
\E^\nu\big(\chisq(\pi_T^\mu\mid\pi_T^\nu)\big)$, divide both sides by
$\var^\nu\big(\gamma_T(X_T)\big) $ to conclude 
\[
R_T \E^\mu\big(\chisq(\pi_T^\mu\mid\pi_T^\nu)\big)\le e^{-c (T-\tau)} \var^\nu(\gamma_0(X_0))
\]
The result follows because $R_T\geq \underline{a}$ (see Remark below). \qed

\begin{remark}[Lower bound for the ratio $R_T$]\label{rem:apdx:RT}
Since $R_T$ is the ratio of expectations of the same random variable
$\clV_T^\nu(\gamma_T)=\chisq(\pi_T^\mu\mid\pi_T^\nu)$ under measures $\sP^\mu$ and $\sP^\nu$
\[
R_T = \frac{\E^\mu\big(\clV^\nu_T(\gamma_T)\big)}{\E^\nu\big(\clV^\nu_T(\gamma_T)\big)}
\ge \mathop{\operatorname{essinf}}_{\omega\in\Omega} \frac{\ud \sP^\mu}{\ud
  \sP^\nu} (\omega) = \mathop{\operatorname{essinf}}_{x\in\bS}\frac{\ud \mu}{\ud \nu}(x) = \underline{a}
\]
An alternative formula for the ratio is as follows:
\[
R_T = \frac{\E^\mu\big(\clV_T^\nu(\gamma_T)\big)}{\E^\nu\big(\clV_T^\nu(\gamma_T)\big)}
=\frac{\E^\nu\big(A_T\clV_T^\nu(\gamma_T)\big)}{\E^\nu\big(\clV_T^\nu(\gamma_T)\big)}
\]
where the change of measure (see~\cite[Sec.~4.5.1]{JinPhDthesis}):
\begin{align*}
\quad A_T:=\frac{\ud \sP^\mu|_{\clZ_T} }{\ud \sP^\nu|_{\clZ_T}} = \exp\Big(&\int_0^T (\pi_t^\mu(h)-\pi_t^{\nu}(h))  \ud I^\mu_t 
\\ &- \half \int_0^T |\pi_t^\mu(h)-\pi_t^{\nu}(h)|^2 \ud t
\Big)
\end{align*}
Now, $\{A_T:T\geq 0\}$ is a non-negative
$\sP^\nu$-martingale with $\E^\nu(A_T)=\E^\nu(A_0)=1$ and therefore, by the martingale convergence theorem, there exists a random variable $A_\infty$ such
that $A_T \stackrel{(T\to\infty)}{\longrightarrow}A_\infty$. It is
possible that an improved asymptotic lower bound for $R_T$ can be
obtained by showing that $\text{essinf}_{\omega\in\Omega}
A_\infty(\omega)>0$.  
\end{remark}

\subsection{Proof of \Lemma{lem:translation}}\label{apdx:fbsde}

\def\tY{\tilde{Y}}
\def\tV{\tilde{V}}

The proof requires showing a Markov property of the optimal control
system~\eqref{eq:optimal_control_system}.  

\newP{Markov property of the optimal control system}
Because $U^\opt=0$, the optimal control system~\eqref{eq:optimal_control_system} is the BSDE
\begin{align}
	-\ud Y_t(x) &= \big((\clA Y_t)(x) + h^\tp(x)V_t(x)\big)\ud t -
	V_t^\tp(x) \ud Z_t, \nonumber\\
	\quad Y_T(x) &= \gamma_T(x)=\frac{\ud \pi_T^\mu}{\ud
		\pi_T^\nu}(x),\;\; x\in\bS, \;\;0\leq t\leq T \label{eq:closed-loop-bsde}
\end{align}
Since the terminal value $Y_T$ is a function of
$\pi_T^\nu$ and $\pi_T^\mu$,  which are both Markov processes, the
Markov property follows from the theory of forward-backward
SDEs~\cite[Chapter 5]{Zhang2017}.  Specifically, for time $s\in[t,T]$,
let $\pi_s^{p,t}$ denote the solution of~\eqref{eq:Kushner} with initial condition $\pi_t=p$.
Then
\begin{align*}
	\pi_s^{\pi_t^\mu(\omega),t}(\omega) = \pi_s^\mu(\omega),  \;\;
	\sP^\mu\text{-a.e.}\;\omega,\;\; t\leq s\leq T \\
	\pi_s^{\pi_t^\nu(\omega),t}(\omega) = \pi_s^\nu(\omega),  \;\;
	\sP^\nu\text{-a.e.}\;\omega,\;\; t\leq s\leq T 
\end{align*}    
Therefore, express
\[
\gamma_T(x)=\frac{\ud \pi_T^\mu}{\ud
	\pi_T^\nu}(x) =\frac{\ud \pi_T^{\pi_t^\mu,t}}{\ud
	\pi_T^{\pi_t^\nu,t}}(x),\;\; x\in\bS
\]
and consider the following BSDE over the time-horizon $[t,T]$:
\begin{align*}
	-\ud \tY_s(x) &= \big((\clA \tY_s)(x) + h^\tp(x)\tV_s(x)\big)\ud s -
	\tV_s^\tp(x) \ud Z_s, \nonumber\\
	\quad \tY_T(x) &= \frac{\ud \pi_T^{\pi_t^\mu,t}}{\ud
		\pi_T^{\pi_t^\nu,t}}(x),\;\; x\in\bS, \;\;t\leq s\leq T \label{eq:closed-loop-bsde-tilde}
\end{align*}
Note that the solution $(\tY,\tV)=\{(\tY_s,\tV_s):t\leq s \leq T\}$
depends on $\pi_t^\nu$ and $\pi_t^\mu$ because of the nature of the
terminal condition. 
The theory of Markov BSDE is used to
assert the following (see~\cite[Ch.~5]{Zhang2017}):

\begin{lemma}[Markov property of the BSDE]\label{lem:lem_MP}
	Let $(Y,V)=\{(Y_t,V_t):0\leq t\leq T\}$ be the solution of~\eqref{eq:closed-loop-bsde}.  Then
	\begin{itemize}
		\item $\tY_s= Y_s$ and $\tV_s= V_s$ for all $t\le s \le T$,
		$\sP^\nu$-a.s..
		\item Given $\pi_t^\mu$ and $\pi_t^\nu$ at time $t$, $\{(\tY_s,\tV_s):t \le s\le T\}$ is independent of $\clZ_t$.
	\end{itemize}
\end{lemma}

\begin{proof}
	See~\cite[Thm.~5.1.3]{Zhang2017}.
\end{proof}

\begin{remark}
	A corollary to the Markov property is the following representation of the
	solution 
	\[
	Y_t = \phi_t(\pi_t^\nu,\pi_t^\mu),\quad 0\leq t\leq T
	\]
	where $\phi_t(\cdot,\cdot)$ is a deterministic function of its
	arguments.  While interesting, the representation is not used in this
	paper.  
\end{remark}

\begin{proof}[Proof of~\Lemma{lem:translation}]
	Based on the Markov property, the following transformation holds $\sP^\nu$-a.s.:
	\begin{align*} 
		&\E^\nu\Big( \int_t^{t+\tau}
		\pi_s^\nu (\Gamma Y_s) + \pi_s^\nu(|V_s|^2)
		\ud s \mid\clZ_t\Big) \\
		&=\E^\nu\Big( \int_t^{t+\tau}
		\pi_s^{\pi_t^\nu,t} (\Gamma \tY_s) + \pi_s^{\pi_t^\nu,t}(|\tV_s|^2)
		\ud s \mid\clZ_t\Big)\\
		&=\E^{\pi_t^\nu}\Big(\int_0^{\tau}
		\pi_{t+s}^{\pi_t^\nu,t} (\Gamma \tY_{t+s}) + \pi_{t+s}^{\pi_t^\nu,t} (|\tV_{t+s}|^2)
		\ud s\Big)\\
		&\stackrel{\text{(PI)}}{\ge} \beta_\tau^{\pi_t^\nu}\clV_t^\nu(\tY_t) = \beta_\tau^{\pi_t^\nu}\clV_t^\nu(Y_t) \;\; (\because, \; Y_t=\tY_t)
	\end{align*}
	where $\beta_\tau^{\pi_t^\nu}$ is now a random number ($\beta_\tau^\rho$ with $\rho=\pi_t^\nu$).
\end{proof}

\subsection{Proof of Theorem~\ref{thm:general-nonexponential-case}}\label{apdx:pf-general}

\newP{Case (i)} By definition of uniform integrability (u.i.), for each $\epsilon>0$, there exists $K$ such that
\[
\E^\nu\big(\clV_T(\gamma_T) \ones_{[\clV_T(\gamma_T) > K]}\big) \le \epsilon,\quad \forall T \ge 0
\]
Therefore,
\begin{align*}
	 \E^\nu\big(e^{-\tau C_N}\clV_T(\gamma_T)\big)  = &
	 \E^\nu\big(e^{-\tau C_N}\clV_T(\gamma_T) \ones_{[\clV_T(\gamma_T)> K]}\big) \\ & + \E^\nu\big(e^{-\tau C_N}\clV_T(\gamma_T) \ones_{[\clV_T(\gamma_T)\le K]}\big)\\
	&\le \epsilon + K\E^\nu\big(e^{-\tau C_N}\big)
\end{align*}
The second term converges to zero from DCT.  Since $\epsilon$ is arbitrary, the result follows.

\newP{Case (ii)} For $\omega \in [C_\infty<\infty]$, $\lim_{k\to \infty} c^{\pi_{k\tau}^\nu}(\omega) = 0$. The result follows because $\{\pi_t^\mu:t\geq 0\}$ is a solution of the SDE~\eqref{eq:Kushner} and therefore a continuous function of time. 

\subsection{Proof of Proposition~\ref{prop:sufficiency}}\label{apdx:pf-main-results-1}

	Suppose any of the three conditions hold. We claim then
	\[
	\text{(claim)} \quad \II^\rho(F) = 0 \implies \VV = 0
	\]
	If the claim is true, the proof is by contradiction.  Suppose $c^\rho = 0$, then by 	Lemma~\ref{prop:crho_positive} there exists $\II^\rho(F) = 0$ such that $\VV =1$ which contradicts the claim.  It remains to prove the claim.  For each of the three cases, the proof is described in the remainder of this section. 

\newP{(i) Ergodic case} At time $t$, let $\rho_t$ denote the probability law of $X_t$ (without
conditioning).  Then because the Markov process is ergodic, for any
$t>0$, the invariant measure $\bmu \ll \rho_t$ (as measures on $\bS$).  W.l.o.g., take $\bS'
= \text{supp}(\rho_t)$ as the new state-space and consider the Markov
process on $\bS'$.  It is again ergodic with the invariant measure
$\bmu\in\clP(\bS')$ and using Defn.~\ref{def:ergodic} of ergodicity,
\begin{equation}\label{eq:ergodic_claim}
	\Gamma f(x)= 0,\;\;\forall x\in \bS' \implies f(x) = c,\;\;\forall x\in\bS'
\end{equation}

Suppose $\II^\rho(F)= 0$.  Because $\sP^\rho\sim\tsP^\rho$,
\begin{align*}
	\E^\rho \left( \int_0^\tau \Gamma Y_t (X_t) \ud t \right) & =0 \\
	\implies
	\Gamma Y_t (X_t) & =0, \; \; \tsP^\rho\text{-a.s.},\;\text{a.e.} \; 0\leq t\leq \tau
\end{align*}
Pick a positive $t$ such that $\Gamma Y_t (X_t) =0$,
$\tsP^\rho\text{-a.s.}$.  Now,  under $\tsP^\rho$,
$X_t\sim\rho_t$, and $X_t$ and $Y_t$ are independent.  Therefore, 
\begin{align*}
	0 = \tE^\rho(\Gamma Y_t(X_t)) = \tE^\rho(\rho_t(\Gamma Y_t)) 
	\implies \rho_t(\Gamma Y_t) = 0, \;\; \tsP^\rho\text{-a.s.}
\end{align*}
Using~\eqref{eq:ergodic_claim},
\begin{equation*}\label{eq:to_prove}
	\rho_t(\Gamma Y_t) = 0, \;\; \tsP^\rho\text{-a.s.} \implies Y_t(x) =
	c_t,\;\;x\in\bS', \; \tsP^\rho\text{-a.s.} 
\end{equation*}
where $c_t$ is $\clZ_t$-measurable.  Then because
$\sP^\rho\sim\tsP^\rho$,
\[
\E^\rho(\clV_t^\rho(Y_t)) \leq \E^\rho( |Y_t(X_t) - c_t|^2) = 0
\]
and the result follows because $\var^\rho(Y_0(X_0))\leq
\E^\rho(\clV_t^\rho(Y_t))$ using~\eqref{eq:estimator-t-variance}.

\begin{remark}
	Note that {\em only} the part of the energy involving the carr\'e du
	champ is used in the proof of the ergodic signal case.  Therefore, for
	an HMM $(\clA,h)$, the
	conclusion depends only upon $\clA$ and holds irrespective of the model $h$ for observations.  
\end{remark}

\newP{(ii) Observable case} The proof is given for HMMs more general than finite state-space: In Defn.~\ref{def:obsvbl}, $\clO$ is now a subspace of $C_b(\bS)$ satisfying the two properties (enumerated as (i) and (ii) in the definition).  In the general setting, an HMM is said to be observable if $\clO$ is dense in $L^2(\rho)$ (written as $\bar{\clO} = L^2(\rho)$).

The key to prove the result is the following Lemma:
\begin{lemma}\label{Lem:obsvbl}
	Suppose $\II^\rho(F) = 0$.  Then for each $ f\in\clO$,
	\[
	\clV_t^\rho(f,Y_t) = 0,  \quad \sP^\rho\text{-a.s.},\;\; \text{a.e.} \; 0\leq t\leq \tau
	\]
\end{lemma}

\begin{proof}
	From the defining relation for $\II^\rho(F)$,
	\[
	\pi_t^\rho (\Gamma Y_t) = 0,\; \clV_t^\rho (h,Y_t) = 0,\;
	\clV_t^\rho (V_t) = 0,\;\;\sP^\rho\text{-a.s.}
	\]
	for $\text{a.e.}\;0\le t\le \tau$.  Using the Cauchy-Schwarz formula
	then for each $ f\in C_b(\bS)$,
	\[
	|\clV_t^\rho (f,V_t)|^2 \leq \clV_t^\rho (f) \clV_t^\rho (V_t) =0 \quad \sP^\rho\text{-a.s.}
	\]
	Similarly, upon using the Cauchy-Schwarz
	formula~\cite[Eq.1.4.3]{bakry2013analysis} for the carr\'e du
	champ operator, 
	\[
	\pi_t^\rho (\Gamma(f,Y_t))= 0, \quad \sP^\rho\text{-a.s.}
	\]
	Based on these, the SDE~\eqref{eq:variance_Yt_f} for the conditional
	covariance simplifies to
	\begin{align*}
		&	\ud \clV_t^\rho (f,Y_t)  = \clV_t^\rho (\clA f,Y_t)\ud t \\
		&\quad	+ \big(\clV_t^\rho (hf,Y_t)  -
		\pi_t^\rho (h)\clV_t^\rho (f,Y_t)\big)^\tp \ud I_t^\rho,\quad 0\le t\le \tau
	\end{align*}
	Therefore, 
	\begin{align*}
		&	\clV_t^\rho (f,Y_t) = 0, \;\;\;0\leq t\leq \tau\\
		&\Longrightarrow \quad \clV_t^\rho (\clA f,Y_t) = 0,\;\clV_t^\rho (hf,Y_t) = 0, \;\;\;0\leq t\leq \tau
	\end{align*}
	Since $\clV_t^\rho (\ones,Y_t) = 0$ for all $t\in[0,\tau]$, the result follows
	from Defn.~\ref{def:obsvbl} of the observable space $\clO$.
\end{proof}

Based on the result in Lemma~\ref{Lem:obsvbl}, the proof of the claim for observable case is completed as follows:

Because $Y_0=\LL_0(F)$ and $\LL_0$ is bounded, $Y_0\in L^2(\rho)$.  
If $\bS$ is finite there is nothing to prove because $\clO=\Re^d$.  
In the case where $\clO\subsetneq \bar{\clO} = L^2(\rho)$, there exists a
sequence $\{f_n\in\clO:n=1,2,\hdots\}$ such that $f_n\to Y_0$ in $L^2(\rho)$. 
From Lemma~\ref{Lem:obsvbl}, for each $n$,
\[
\clV_0^\rho(f_n,Y_0) = 0, \quad \sP^\rho\text{-a.s.}
\]  
Therefore, 
\begin{align}\label{eq:tricky}
	\var^\rho(Y_0(X_0)=\clV_0^\rho(Y_0)  &=\clV_0^\rho (Y_0-f_n,Y_0)
\end{align}
and letting $n\to\infty$, because $f_n\to Y_0$, $\var^\rho(Y_0(X_0))=0$
using the Cauchy-Schwarz.

\newP{(iii) Detectable case} As shown in the ergodic case, if $\II^\rho(F) = 0$ then $\Gamma Y_t(x) = 0$ for all $x\in\bS'$, and therefore $Y_t \in S_0$. If the system $(\clA,h)$ is detectable, then this implies $Y_t \in \clO$. By Lemma~\ref{Lem:obsvbl}, $\E^\rho\big(\clV_t^\rho(Y_t)\big) = 0$ and the claim follows.

\subsection{Proof of Theorem~\ref{thm:finite-case}}\label{apdx:pf-finite}

Let $\delta_s$ denote the Dirac delta probability measure with support at $s\in\bS$. Denote
\begin{align*}
	{\cal N}_0 &= \{\delta_s:s \in\bS\}\\
	{\cal N}_\epsilon &= \{\rho \in \clP(\bS) : \rho(s) > 1-\epsilon \;\text{for one}\; s\in\bS\}
\end{align*}
${\cal N}_0$ is a subset of $\clP(\bS)$ comprising of $d$ Dirac delta measures ($d$ vertices of the probability simplex).  ${\cal N}_\epsilon$ is the $\epsilon$-neighborhood of ${\cal N}_0$.  We claim that ${\cal N}= {\cal N}_0$.  Assuming the claim to be true, the proof steps to show \Theorem{thm:finite-case} are as follows:

\newP{Step 1} Show that $\{\clV_T^\nu(\gamma_T) : T\geq 0\}$ is $\sP^\nu$-u.i.  This is because of the formula for the forward map (See Rem.~\ref{rm:forward-map}):
\[
\max_{x\in\bS} |\gamma_T(x)| \leq \frac{\bar{a}}{\underline{a}},\;\; \sP^\nu-\text{a.s.} 
\]

\newP{Step 2} Show that on $[C_\infty < \infty]$, $\clV_T^\nu(\gamma_T) \to 0$, $\sP^\nu$-a.s.. 
This is where the assumption of detectability is used.  From Thm.~\ref{thm:general-nonexponential-case}, on $[C_\infty < \infty]$, $c^{\pi_T^\nu(\omega)} \to 0$ $\sP^\nu$-a.s..  Because $c^\rho>0$ and $\rho\mapsto c^\rho$ is continuous for points in the interior of $\clP(\bS)$ (\Lemma{prop:crho_positive}), $\pi_T^\nu(\omega)$ eventually escapes every compact set in the interior of $\clP(\bS)$.  For $d=2$, this means that for each $\epsilon>0$, there exists a $\bar{T} = \bar{T}(\omega,\epsilon)$ such that $\pi_T^\nu(\omega) \in {\cal N}_\epsilon$ for all $T>\bar{T}$. It is a straightforward estimate then to show that
\[
\clV_T^\nu(\gamma_T) \leq 4\epsilon \Big( \frac{\bar{a}}{\underline{a}} \Big)^2,\quad T \geq \bar{T}
\]
Since $\epsilon$ is arbitrary, it follows that $\clV_T^\nu(\gamma_T) \to 0$, $\sP^\nu$-a.s.. 

\newP{Step 3} From~\eqref{eq:var_contractive},
\begin{align*}
	& \var^\nu(Y_0(X_0))  \leq \E^\nu \left( e^{-\tau C_N} \clV_T^\nu(\gamma_T) \right) = \\
	&  \E^\nu \left( \ones_{[C_\infty = \infty]} e^{-\tau C_N} \clV_T^\nu(\gamma_T) \right) 
	+ \E^\nu \left( \ones_{[C_\infty < \infty]} e^{-\tau C_N} \clV_T^\nu(\gamma_T) \right)
\end{align*}
The first of these terms goes to zero because $ e^{-\tau C_N} \to 0$ $\sP^\nu$-a.s. on $[C_\infty = \infty]$.  The second of these terms goes to zero from step 2.

\newP{Proof of the claim ${\cal N}= {\cal N}_0$} For $\rho \in \clP(\bS) \setminus {\cal N}_0$, pick a function $f\in \Re^d$ such that $\rho(f) = 0 $ and $\rho(f^2)=1$.  Such a $f$ always exists: Pick two points $s_1,s_2 \in \bS$ such $\rho(s_1)>0$ and $\rho(s_2)>0$.  Set
\[
f(s) = \begin{cases} \frac{a}{\rho(s_1)} & s=s_1\\
	-\frac{a}{\rho(s_2)} & s=s_2\\
	0 & s\in \bS\setminus \{s_1,s_2\}
\end{cases}
\]
where $a=\sqrt{\frac{\rho(s_1)\rho(s_2)}{\rho(s_1)+\rho(s_2)}}$. Now solve the forward-in-time linear ordinary differential equation
\begin{align*}
	\frac{\ud Y_t}{\ud t}(x) &= - (\clA Y_t)(x) +  h^\tp (x) \clV_t^\rho(h,Y_t) \\
	Y_0(x) &= f(x),\;\; x\in \bS, \;\;0\leq t\leq \tau
\end{align*}
This is finite-dimensional linear system with uniformly bounded random coefficients.  So, it admits a well-defined bounded solution at time $t=\tau$.  Denote the solution $Y_\tau = F$. Because $F$ is bounded, $F\in \mathbb{H}^\rho_\tau$.  Now, consider dual optimal control system~\eqref{eq:optimal_control_system} with $Y_T=F$.  Then by uniqueness of the solution, $V=0$ and $Y_0=f$.  By construction, $\VV=1$. This shows that ${\cal N}\subset {\cal N}_0$.  To show that ${\cal N}= {\cal N}_0$, note $\rho(f^2) = \rho(f)^2 = |f(s)|^2$ for $\rho = \delta_s$.  Therefore, $\VV=0$ for $\rho \in {\cal N}_0$.

\begin{remark}\label{rem:continuity_appdx}
  For $d>2$, it is still true (in step 2) that $\pi_T^\nu(\omega)$ eventually escapes every compact set in the interior of $\clP(\bS)$.  However, a subsequential limit could be to a point on the boundary.  To extend the proof to $d>2$ requires one to show that $\rho \mapsto c^\rho$ is continuous at the boundary points $\rho\in\clP(\bS)\setminus {\cal N}$.
\end{remark}

\bibliographystyle{IEEEtran}
\bibliography{../../../bibfiles/_master_bib_jin.bib,../../../bibfiles/jin_papers.bib,../../../bibfiles/extrabib.bib,../../../bibfiles/estimator_controller.bib}

\begin{IEEEbiography}[{\includegraphics[width=1in,height=1.25in,clip,keepaspectratio]{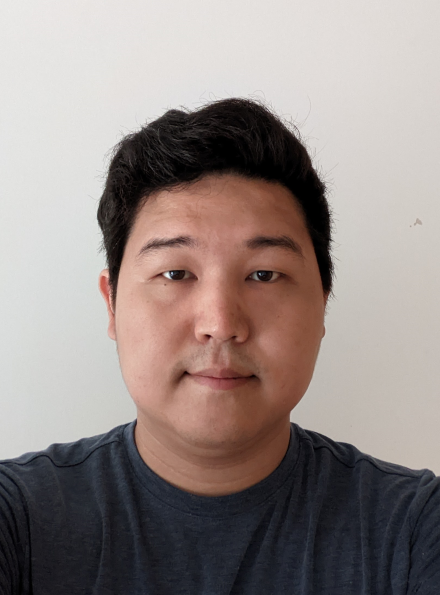}}]{Jin Won Kim} received the Ph.D. degree in Mechanical Engineering from University of Illinois at Urbana-Champaign, Urbana, IL, in 2022.
	He is now a postdocdoral research scientist in the Institute of Mathematics at the University of Potsdam.
	His current research interests are in nonlinear filtering and stochastic optimal control.
	He received the Best Student Paper Award at the IEEE Conference on Decision and Control 2019.
\end{IEEEbiography}

\begin{IEEEbiography}[{\includegraphics[width=1in,height=1.25in,clip,keepaspectratio]{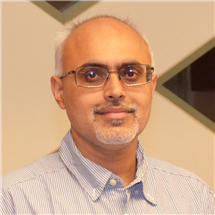}}]{Prashant G. Mehta} received the Ph.D. degree in Applied Mathematics from Cornell University, Ithaca, NY, in 2004.
	He is a Professor of Mechanical Science and Engineering at the University of Illinois at Urbana-Champaign.
	Prior to joining Illinois, he was a Research Engineer at the United Technologies Research Center (UTRC). His current research interests are in nonlinear filtering. He received the Outstanding Achievement Award at UTRC for his contributions to the modeling and control of combustion instabilities in jet-engines. His students received the Best Student Paper Awards at the IEEE Conference on Decision and Control 2007, 2009 and 2019, and were finalists for these awards in 2010 and 2012. In the past, he has served on the editorial boards of the ASME Journal of Dynamic Systems, Measurement, and Control and the Systems and Control Letters. He currently serves on the editorial board of the IEEE Transactions on Automatic Control. 	
\end{IEEEbiography}

\end{document}